\documentclass[11pt,a4paper,reqno]{amsart}
\usepackage{color}
\usepackage{amsfonts,amsmath,amssymb,amsxtra,url,float} 
\usepackage[colorlinks,linkcolor=RoyalBlue,anchorcolor=Periwinkle,citecolor=Orange,urlcolor=Green]{hyperref} 
\usepackage[usenames,dvipsnames]{xcolor} 
\usepackage{enumitem}
\usepackage{geometry} 
\usepackage{graphicx} 
\usepackage{subfigure} 
\usepackage{tikz}\usetikzlibrary{matrix}\usetikzlibrary{trees}
\usepackage[all]{xy} 
\usepackage{mathdots} 
\usepackage{yhmath} %
\usepackage{dsfont} 
\usepackage{cite}
\usepackage{mathrsfs} 
\usepackage{multicol} 
\numberwithin{figure}{section}

\usepackage{marginnote} 
\usepackage{graphicx} 
\usepackage{multicol} 

\newtheorem{theorem}{Theorem}[section]
\newtheorem{lemma}[theorem]{Lemma}
\newtheorem{corollary}[theorem]{Corollary}
\newtheorem{main theorem}[theorem]{Main Theorem}
\newtheorem{proposition}[theorem]{Proposition}
\newtheorem{definition}[theorem]{Definition}
\newtheorem{construction}[theorem]{Construction}
\newtheorem{remark}[theorem]{Remark}
\newtheorem{example}[theorem]{Example}

\newtheorem{question}[theorem]{Question}

\usetikzlibrary{arrows}

\numberwithin{equation}{section}

\def\<{\langle} 
\def\>{\rangle} 
\def\NN{\mathbb{N}} 
\def\ZZ{\mathbb{Z}} 

\newcommand{\Pic}{F{\tiny{IGURE}}\ }
\newcommand{\modcat}{\mathsf{mod}}
\newcommand{\projcat}{\mathsf{proj}}

\newcommand{\ind}{\mathsf{ind}}
\newcommand{\str}{\mathsf{Str}}
\newcommand{\band}{\mathsf{Band}}
\newcommand{\rad}{\mathrm{rad}}
\newcommand{\kk}{\mathds{k}} 
\newcommand{\Q}{\mathcal{Q}} 
\newcommand{\I}{\mathcal{I}} 
\newcommand{\M}{\mathds{M}}
\newcommand{\Hom}{\mathrm{Hom}}
\newcommand{\End}{\mathrm{End}}
\newcommand{\Ext}{\mathrm{Ext}}
\newcommand{\Tor}{\mathrm{Tor}}
\newcommand{\Irr}{\mathrm{Irr}}
\newcommand{\C}{\mathscr{C}}

\newcommand{\scrG}{\mathscr{G}}
\newcommand{\frakv}{\mathfrak{v}}

\newcommand{\Gproj}{\mathsf{G}\text{-}\mathsf{proj}}
\newcommand{\NTIGproj}{\mathsf{NTIG}\text{-}\mathsf{proj}}
\newcommand{\perpath}{\mathsf{per}.\mathsf{path}}
\newcommand{\CMA}{\mathrm{CMA}}

\newcommand{\pd}{\mathrm{proj.dim}}
\newcommand{\id}{\mathrm{inj.dim}}
\newcommand{\gldim}{\mathrm{gl.dim}}
\newcommand{\HH}{\mathrm{H}}
\newcommand{\hl}{\mathrm{hl}}
\newcommand{\hw}{\mathrm{hw}}
\newcommand{\hr}{\mathrm{hr}}

  \def\defines{\it}

\begin{document}

\title[The CM-Auslander algebras of string algebras]{The Cohen-Macaulay Auslander algebras of string algebras}
\thanks{$^{\ast}$Corresponding author.}
\author{Yu-Zhe Liu}
\address{Y.-Z. Liu, School of Mathematics and statistics, Guizhou University, Guiyang 550025,
P. R. China}
\email{liuyz@gzu.edu.cn \text{/} yzliu3@163.com}

\author{Chao Zhang$^{\ast}$}
\address{C. Zhang, School of Mathematics and statistics, Guizhou University, Guiyang 550025,
P. R. China}
\email{zhangc@amss.ac.cn}





\begin{abstract}
  The Cohen-Macaulay Auslander algebra of any string algebra is explicitly constructed in this paper.
  Furthermore, we show that a class of special string algebras, which are called to be string algebras with {\it G-condition}, are representation-finite
  if and only if their Cohen-Macaulay Auslander algebras are representation-finite. Finally,
  the self-injective dimension of gentle algebras is characterized using their Cohen-Macaulay Auslander algebras.
\end{abstract}

\maketitle

\section{Introduction}

Gorenstein projective (=G-projective) modules were introduced by Auslander in \cite{A1967, AB1969},
which are natural generalization of projective modules in order to provide some homological invariants of Noetherian rings.
In \cite{AB1989, B1986}, Auslander and Buchweitz gave a more categorical approach to the study of G-projective modules.
After that, algebraists found some important applications of G-projective modules in commutative algebra, algebraic geometry and relative homological algebra \cite[etc]{B2005, B2011, C2018, EJ2000, IKM2011, LZ2010, LuoZ2013, XZ2012, Zhang2013Gorenstein}.
Recently, it is shown that a module $X$ over Gorenstein ring $A$ is Gorenstein-projective
if and only if $\Ext_A^{\ge 1}(X, \mathrm{Add}P)=0$ \cite[Corollary 11.5.3]{EJ2000}.
Furthermore,  Beligiannis proved that, for any Gorenstein algebra $A$, a module $X$ is G-projective if and only if $\Tor^A_{\ge 1}(D({_AA}), X)=0$ \cite[Proposition 3.10]{B2005}.
If $A \cong \kk\Q/\I$ is a monomial algebra, then G-projective modules can be described in terms of PPSs on the bounded quiver $(\Q, \I)$ of $A$ \cite{C2018}.
Moreover, one can construct Gorenstein-projective modules over triangular matrix algebra
$\left(\begin{smallmatrix} A & M \\ 0 & A' \end{smallmatrix}\right)$,
where $M = {_AM_{A'}}$ is an $(A, A')$-bimodule
(see \cite{LZ2010, IKM2011, XZ2012, LuoZ2013, Zhang2013Gorenstein}).

The Auslander algebra of an artin algebra $A$, defined in \cite{A1971} as the endomorphism algebra of all indecomposable $A$-modules, is characterized as algebras of global dimension at most 2
and dominant dimension at least 2. To study the algebras of finite Cohen-Macaulay type (=finite CM-type),  Beligiannis introduced the Cohen-Macaulay Auslander algebra (=CM-Auslander algebra) $B$ of an algebra $A$, in a similar way with the Auslander algebra,
as the endomorphism algebra of the direct sum of all indecomposable G-projective left (resp.  right) $A$-modules in \cite{B2011},  which is also called the relative Auslander algebra.
CM-Auslander algebra has many important properties.  It is proved that the homological dimensions of $A$ are closely related to that of $B$ (see \cite[Section 6]{B2011}).
If $A$ is CM-finite, then $A$ is Gorenstein if and only if $\gldim B < \infty$
\cite{LZ2010CMAuslander}. Moreover,
the CM-Auslander algebra of CM-finite algebra is CM-free, see \cite[Theorem 4.5]{KZ2016}.
Many other properties of $A$ can be reflected by its CM-Auslander algebra,
such as \cite{CL2017, CL2019, Haf2018CMAuslander, Pan2012CMAuslander, Qin2020CMAuslander}.
Then it is natural to consider the question as follows.
\begin{question} \label{Q1}
How to represent the CM-Auslander algebra for an algebra of finite CM-type?
\end{question}
In \cite{CL2017, CL2019}, Chen and Lu provided a method to construct the CM-Auslander algebras of all (skew-)gentle algebras
and showed that (skew-)gentle algebras are representation-finite if and only if
their CM-Auslander algebras are representation-finite.
Thus, a further question as follows can be raised naturally.
\begin{question} \label{Q2}
If $A$ {\rm(}resp.  $B${\rm)} is a representation-finite algebra,
then is $B$ {\rm(}resp.  $A${\rm)} a representation-finite algebra?
\end{question}

Throughout this paper, $\kk$ is an algebraically closed field
and any finite dimensional algebra $A$ is isomorphic to $\kk\Q/\I$,
where $\Q$ is a finite quiver and $\I$ is an admissible ideal.
An quiver $\Q$ is a quadruples $(\Q_0, \Q_1, s, t)$, where $\Q_0$ and $\Q_1$ are finite sets of vertices and arrows,
$s$, $t$ are functions sending any arrow in $\Q_1$ to its starting point and ending point, respectively.
For arbitrary two arrows $a_1$ and $a_2$ with $t(a_1)=s(a_2)$,
the composition of $a_1$ and $a_2$ is denoted by $a_1a_2$. All modules over $A$ are finite generated right $A$-modules, we denote by $\modcat A$ the finite generated right $A$-module category of $A$.
For any full subcategory $\mathcal{C}$ of $\modcat A$, we denote by $\ind(\mathcal{C})$ the full subcategory of $\mathcal{C}$ containing all indecomposable objects lying $\mathcal{C}$.
For any module $M \in \modcat A$,  denote by $\pd M$ and $\id M$ the projective dimension and injective dimension of $M$, respectively.
Moreover, $\gldim A$ is the global dimension of $A$.

The present paper is devoted to describing all G-projective modules over string algebras,
and calculating all irreducible morphisms in $\Gproj A$ by the description of Auslander-Reiten quiver of string algebra $A$ provided in \cite{BR1987},
where $\Gproj A$ is the full subcategory of $\modcat A$ induced by indecomposable G-projective $A$-module.
To be more precise, all G-projective modules over string algebra are provided by some oriented cycles of $(\Q, \I)$,
and the CM-Auslander algebras of string algebras can be characterized as in the following theorem.

\begin{theorem} \label{thm:1} {\rm (Theorem \ref{thm:the CMA of string})}
Let $A$ be a string algebra and $B$ its CM-Auslander algebra.
Then $B \cong A^{\CMA} = \kk\Q^{\CMA}/\I^{\CMA}$, where $(\Q^{\CMA}, \I^{\CMA})$ is the bounded quiver given by Construction \ref{const:CMA}.
\end{theorem}

\noindent The above theorem provides a partial answer of Question \ref{Q1} in the case of string algebras.
Furthermore, one can judge whether the CM-Auslander algebra of a string algebra $A$
is representation-finite or not.
As an application, we find that Question \ref{Q2} does not hold in general (see Examples \ref{fig:CMA of A in ex} and \ref{fig:CMA of A in ex:self-inj}).
However, it is true for string algebras with {\it G-condition}
i.e., the string algebras whose all oriented cycles providing at least one G-projective module are gentle relation cycles
(see subsection \ref{subsec:G-cond}),
as shown in the following theorem.

\begin{sloppypar} 
\begin{theorem} \label{thm:2} {\rm (Theorem \ref{thm:G-condition})}
Let $A$ be a string algebra satisfying G-condition and $A^{\CMA}$ its CM-Auslander algebra.
Then $A$ is representation-finite if and only if $A^{\CMA}$ is representation-finite.
\end{theorem}
\end{sloppypar}

Paralell to the representation type of finite-dimensional algebras, the derived representation type mainly consider the classification and disturbutation of objects in derived category, which was pioneered by Vossieck \cite{Vo01}. He defined the derived-discrete algebras and classified them, which contain gentle one-cycle algebras not satisfying the clock condition (namely, gentle algebras without homotopy bands) as one type. The derived representation type was also studied in other related works, see for example \cite{ZH16, Zh16a, Zh16b}.  As another application, the present paper observes the relation of derived representation type of a gentle algebra and its CM-Auslander algebra. 

This paper is organized as follows.
In Section \ref{sec:string and gentle}, some preliminaries on the string algebras and gentle algebras are recalled.
The section \ref{sec:CMA} explicitly constructs the CM-Auslander algebras of any string algebra, and describes the representation type of any string algebra by the CM-Auslander algebra.
Moreover, a partial answer of Question \ref{Q1} in the case of string algebra is provided, and Question \ref{Q2} is negated (see Remark \ref{rmk:repr-type}).
As an application, the last section shows some results of gentle algebras,
which are originally proved in \cite{CL2019} and \cite{B2011}.

\section{String algebras and gentle algebras} \label{sec:string and gentle}

Let $(\Q, \I)$ be a bounded quiver. A {\defines circuit} $\C$ (of length $n$)
is a path of length $n$ on the underlying graph of $\Q$ such that its starting point and ending point coincide.
A circuit $\C$ is a {\defines cycle} of $(\Q, \I)$ if any two arrows on $\C$ are unequal.
{A cycle  $\C=c_1\cdots c_n$ is an {\defines oriented cycle} if any $c_i$ and $c_{i+1}$ are arrows satisfying $t(c_{i})=s(c_{i+1})$ ($1\le i\le n-1$).
Furthermore, an oriented cycle $\C = c_1\cdots c_n$ is called a {\defines relation cycle} if there are generators $r_0, r_1, \cdots, r_{t-1}$ of $\I$ such that
\begin{itemize}
  \item[(1)] $r_0$, $r_1$, $\cdots$, $r_{t-1}$ are paths on $\C$;
  \item[(2)] for any $0\le i\le t-1$, the paths $r_{\overline{i}}$ and $r_{\overline{i+1}}$ overlap, i.e., $t(r_{\overline{i}})$ lies between $s(r_{\overline{i+1}})$ and $t(r_{\overline{i+1}})$,
  and $s(r_{\overline{i+1}})$ lies between $s(r_{\overline{i}})$ and $t(r_{\overline{i}})$ where $\overline{i}$ is the representative of $i$ modulo $t$.
\end{itemize} }
\noindent A {\defines cycle {\rm(}resp. circuit{\rm)} without relation}
is a cycle $\C$ (resp. circuit) such that all paths on $\C$ are not in $\I$.

For any two paths $p = a_1a_2\cdots a_l$ of length $l$ and $q = a_{i+1}a_{i+2}\cdots a_{i+m}$ of length $m$ ($1 \le i < l \le i+m$),
we denote by $p \sqcap q = a_{i+1}a_{i+2}\cdots a_l$ and $p \sqcup q = a_1a_2\cdots a_{i+m}$.

\subsection{String algebras}

A quiver $(\Q, \I)$ is said to be {\defines a string quiver} if
\begin{itemize}
  \item any vertex of $\Q$ is the starting point and ending point of at most two arrows.

  \item for each arrow $\alpha:x\to y$,
  there is at most one arrow $\beta$ whose starting point (resp. ending point) is $y$ (resp. $x$)
  such that $\alpha\beta\notin\I$ (resp. $\beta\alpha\notin\I$).

  \item $\I$ is admissible, and it is generated by paths of length $\ge 2$.
\end{itemize}
Moreover, a string quiver $(\Q, \I)$ is {\defines gentle} if:
\begin{itemize}
  \item for each arrow $\alpha: x\to y$,
  there is at most one arrow $\beta$ whose starting point (resp. ending point) is $y$ (resp. $x$)
  such that $\alpha\beta\in\I$ (resp. $\beta\alpha\in\I$).

  \item all generators of $\I$ are paths of length two.
\end{itemize}

\begin{definition}\rm \label{def:string}
A {\defines string {\rm(}resp.  gentle{\rm)} algebra} is a finitely dimensional $\kk$-algebra $A$
which is isomorphic to $\kk\Q/\I$, where $(\Q, \I)$ is a string {\rm(}resp.  gentle{\rm)} quiver.
\end{definition}

\begin{remark} \rm
Any gentle algebra is a string algebra. Moreover, all string algebras are finitely dimensional in Definition \ref{def:string},
thus the quiver of string algebra is {\defines finite}, i.e., $\sharp\Q_0<\infty$ and $\sharp\Q_1<\infty$.
\end{remark}

For any arrow $a\in \Q_1$, we denote by $a^{-1}$ the {\defines formal inverse} of $a$ with $s(a^{-1})=t(a)$ and $t(a^{-1})=s(a)$,
and let $\Q_1^{-1}:=\{a^{-1}\mid a\in \Q_1\}$ be the set of all formal inverses of arrows. The formal inverse extend naturally to any path, i.e.,
for any path $\wp=a_1a_2\cdots a_l$ $(\notin\I)$, the {\defines formal inverse path} $\wp^{-1} = a_{l}^{-1}a_{l-1}^{-1}\cdots a_1^{-1}$.
Notice that for any trivial path $e_v$ corresponding to $v\in \Q_0$, we define $e_v^{-1} = e_v$.

\begin{definition}\rm
A {\defines string} is a sequence $s=(\wp_1, \wp_2, \ldots, \wp_n)$ such that:
\begin{itemize}
  \item[(1)]
    for any $1\le i\le n$, one of $\wp_i$ and $\wp_i^{-1}$ is a path of $(\Q, \I)$;

  \item[(2)]
    if $\wp_i$ is a path (resp. formal inverse of a path), then $\wp_{i+1}$ is a formal inverse of a path (resp. path);

  \item[(3)]
    $t(\wp_{i})=s(\wp_{i+1})$ holds for all $1\le i\le n-1$.
\end{itemize}
Furthermore, A {\defines band} $b=(\wp_1, \wp_2, \ldots, \wp_n)$ is a string such that
\begin{itemize}
  \item[(4)] $t(\wp_n)=s(\wp_1)$;
  \item[(5)] $b$ is not a non-trivial power of some string, i.e.,
    there is no string $s$ such that $b=s^m$ for some $m\ge 2$;
  \item[(6)] $b^2$ is a string.
\end{itemize}
\end{definition}

\begin{remark}\rm
A vertex $v\in \Q_0$ on the string $s=(\wp_1, \wp_2, \cdots, \wp_n)$
is called a {\defines self-intersection} of $s$, if there are two integers $1\le i\ne j\le n$
such that $v$ is both a vertex on $\wp_i$ and a vertex on $\wp_j$,
i.e., $v$ is a vertex where $s$ crosses itself at least two times.
\end{remark}


Two strings $s=(\wp_1, \ldots, \wp_n)$ and $s'=(\wp_1', \ldots, \wp'_n)$
are called to be {\defines equivalent}, denoted by $s\simeq s'$, if $s'=s^{-1}$
(i.e., $\wp_{n-i+1}'^{-1} = \wp_i$ for all $1\le i\le n$).
Two bands $b=(\wp_1, \ldots, \wp_n)$ and $b'=(\wp_1', \ldots, \wp'_n)$
are called to be {\defines equivalent}, denoted by $b\simeq b'$, if one of following holds:
\begin{itemize}
  \item
    there is an integer $1\le j\le n$ such that $\wp_{\overline{i}}' = \wp_{\overline{i+j}}$ holds for all $1\le i\le n$;

  \item
    there is an integer $1\le j\le n$ such that $\wp_{\overline{i}}'^{-1} = \wp_{\overline{(n+1-i)+j}}$ holds for all $1\le i\le n$,
\end{itemize}
where $\overline{x}$ is the reprsentative element of $x$ modulo $n$ in the set $\{1, 2, \cdots, n\}$.
We denote by $\str(A)$ (resp. $\band(A)$) the set of all equivalent classes of strings (resp. bands).

\subsection{String modules and band modules}

\begin{definition}[String module and band module] \rm \cite[Section 3]{BR1987}
Let $s=a_1\cdots a_n$ be a string of length $n$.
A {\defines string module} $M$ associated with $s$ is an indecomposable module defined as follows
\begin{itemize}
  \item 
    for each vertex $v$ of $s$,
    $\dim_{\kk}M e_v$ is the number of the times of $v$ traversed by $s$ or zero if otherwise; and

  \item the action of an arrow $\alpha\in\Q_1$ on $M$ is the identity morphism if $\alpha$ is an arrow of $s$ or zero if otherwise.
\end{itemize}

Let $b=a_1\cdots a_n$ be a band of length $n$.
A {\defines band module} $B$ corresponding to the pair $(b, \varphi)$ of $b$ and $\varphi\in \mathrm{Aut}(\kk^n)$ is an indecomposable module obtained by following:
\begin{itemize}
  \item $\varphi$ is an indecomposable $\kk$-linear automorphism of $\kk^n$ with $n > 0$
    (that is, $\varphi \in \mathrm{Aut}(\kk^n)$ is a Jordan block because $\kk$ is algebraically closed);

  \item  
    for each vertex $v$ of $s$, $M e_v$ is $\kk$-linear isomorphic to $\kk^{mn}$,
    where $m$ is the number of t he times of $v$ traversed by $s$ or zero if otherwise; and

  \item the action of an arrow $\alpha \in \Q_1$ on $B$ is the identity morphism if $\alpha = a_i$ for $1\le i\le n-1$ or $\varphi$ if $\alpha = a_n$ (thus $B\alpha = 0$ if $\alpha$ is not an arrow of $b$).
\end{itemize}
\end{definition}

The following theorem shows that any indecomposable module over string algebra is either a string module or a band module.

\begin{theorem} \label{thm:indmod}
Let $A$ be a string algebra. Then there is a bijection
\[ \M: \str(A) \cup (\band(A)\times\mathscr{J}) \to \ind(\modcat(A)), \]
where $\mathscr{J}$ is the set of all Jordan block with non-zero eigenvalue.
\end{theorem}

\section{The Cohen-Macaulay Auslander algebras of string algebras} \label{sec:CMA}

The present section is devoted to describing non-trivial indecomposable Gorenstein-projective modules (= NTIG-projective modules) over string algebras.

\subsection{Perfect pairs and perfect paths}

Obviously, all projective modules are G-projective.
We say a G-projective module is {\defines trivial} if it is projective.
In \cite{C2018}, all NTIG-projective modules over monomial algebra are characterized in terms of perfect paths. Here, monomial algebras are those algebras admitting a presentation $\kk\Q/\I$ such that $\I$ is generated by paths of length $\ge 2$.
In this subsection, we review the relevant concepts and results in \cite{C2018}.

\begin{definition}[{Perfect pairs and perfect paths}] \rm
Let $(\Q, \I)$ be a bounded quiver of some monomial algebra.
\begin{itemize}
\item[(1)] \cite[Definition 3.3]{C2018}
A {\defines perfect pair} on $(\Q, \I)$ of monomial algebra $\Lambda$ is a pair of two non-trivial paths $(\wp_1,\wp_2)$ such that the following conditions hold.
\begin{itemize}
  \item[(i)] $\wp_1$ and $\wp_2$ are non-trivial paths such that $t(\wp_1)=s(\wp_2)$ and $\wp_1\wp_2=0$.

  \item[(ii)] If there exists a non-trivial path $\wp$ satisfying $t(\wp_1)=s(\wp)$ (resp. $t(\wp)=s(\wp_2)$) such that $\wp_1\wp=0$ (resp. $\wp\wp_2=0$),
    then $\wp_2$ (resp. $\wp_1$) is a subpath of $\wp$.
\end{itemize}

\item[(2)] \cite[Definition 3.7]{C2018}
A {\defines perfect path sequence} (=PPS) is a sequence of paths of the form $(p=\wp_1, \wp_2, \cdots, \wp_t=p)$ such that any pair $(\wp_i, \wp_{i+1})$ is a perfect pair for $i=1, 2, \cdots, t$. Any path in a perfect path sequence is called to be a {\defines perfect path}.
\end{itemize}
\end{definition}

\begin{remark} \rm
If $(\wp_1, \wp_2, \cdots, \wp_{t-1}, \wp_1)$ is a PPS, then so is $(\wp_2, \cdots, \wp_{t-1}, \wp_1, \wp_2)$, and all $\wp_i$ are paths on the same relation cycle $\C$.
\end{remark}

To describe all NTIG-projective modules over a string algebra $A=\kk\Q/\I$,
we need the following lemma, which depicts the oriented cycle of the bounded quiver $(\Q, \I)$ of $A$ contains at least one PPS.

\begin{lemma} \label{lemm:per.path}
Let $(p=\wp_0, \wp_1, \cdots, \wp_{t-1}=p)$ be a PPS on the bounded quiver of string algebra $A=\kk\Q/\I$,
where $\wp_{\overline{i}} = \alpha_{\overline{i},1} \cdots \alpha_{\overline{i}, l_{\overline{i}}}$ and $\overline{i}$ equals to $i$ modulo $t$.
Then the following statements hold.
\begin{itemize}
  \item[\rm(1)] If there exists arrow $\alpha$ with starting point {\rm(}resp.  ending point{\rm)} is $t(\wp_{\overline{i}}) = s(\wp_{\overline{i+1}}) \in \Q_0$,
      then $\wp_{\overline{i}} \alpha \ne 0$ {\rm(}resp.  $\alpha \wp_{\overline{i+1}} \ne 0${\rm)}.
      In this case, we have $\ell(\wp_{\overline{i}}) = 1 = \ell(\wp_{\overline{i+1}})$.
  \item[\rm(2)] All compositions $\wp_{\overline{i}}\wp_{\overline{i+1}}$ must be generators of $\I$ {\rm(}see \Pic \ref{fig:per.path}{\rm)};
\end{itemize}
\end{lemma}

\begin{figure}[H]
\begin{tikzpicture}
\shade[left color=orange!50, right color=orange][opacity=0.5][shift={(-3, 0)}] (1.5,0) ellipse (1.5cm and 0.25cm);
\draw[shift={(-3, 0)}][->] (0.2,0) node[left]{$\bullet$} -- (0.8,0) node[right]{$\bullet$};
\draw[shift={(-2, 0)}][->] (0.2,0) -- (0.8,0) node[right]{$\bullet$} [dash pattern=on 2pt off 2pt];
\draw[shift={(-1, 0)}][->] (0.2,0) -- (0.8,0) node[right]{$\bullet$};
\draw[shift={(-3, 0)}][orange] (1.5,-0.2) node[below]{$\wp_{\overline{i-1}}$};
\shade[left color=pink, right color=red][opacity=0.5] (1.5,0) ellipse (1.5cm and 0.25cm);
\draw[->] (0.2,0) -- (0.8,0) node[right]{$\bullet$};
\draw (0.5,0) node[above]{$\alpha_{\overline{i},1}$};
\draw[shift={( 1, 0)}][->] (0.2,0) -- (0.8,0) node[right]{$\bullet$} [dash pattern=on 2pt off 2pt];
\draw[shift={( 2, 0)}][->] (0.2,0) -- (0.8,0) node[right]{$\bullet$};
\draw[shift={( 2, 0)}] (0.5,0) node[above]{$\alpha_{\overline{i},l_{\overline{i}}}$};
\draw[red] (1.5,-0.2) node[below]{$\wp_{\overline{i}}$};
\shade[left color=blue, right color=blue!50][opacity=0.5][shift={( 3, 0)}] (1.5,0) ellipse (1.5cm and 0.25cm);
\draw[shift={( 3, 0)}][->] (0.2,0) -- (0.8,0) node[right]{$\bullet$};
\draw[shift={( 3, 0)}] (0.5,0) node[above]{$\alpha_{\overline{i+1},1}$};
\draw[shift={( 4, 0)}][->] (0.2,0) -- (0.8,0) node[right]{$\bullet$} [dash pattern=on 2pt off 2pt];
\draw[shift={( 5, 0)}][->] (0.2,0) -- (0.8,0) node[right]{$\bullet$};
\draw[shift={( 5, 0)}] (0.5,0) node[above]{$\alpha_{\overline{i+1},l_{\overline{i+1}}}$};
\draw[shift={( 3, 0)}][blue] (1.5,-0.2) node[below]{$\wp_{\overline{i+1}}$};
\shade[left color=cyan, right color=cyan!50][opacity=0.5][shift={( 6, 0)}] (1.5,0) ellipse (1.5cm and 0.25cm);
\draw[shift={( 6, 0)}][->] (0.2,0) -- (0.8,0) node[right]{$\bullet$};
\draw[shift={( 7, 0)}][->] (0.2,0) -- (0.8,0) node[right]{$\bullet$} [dash pattern=on 2pt off 2pt];
\draw[shift={( 8, 0)}][->] (0.2,0) -- (0.8,0) node[right]{$\bullet$};
\draw[shift={( 6, 0)}][cyan] (1.5,-0.2) node[below]{$\wp_{\overline{i+2}}$};
\shade[left color=orange, right color=red][dotted][line width=2pt][opacity=0.5][shift={(-3,0)}]
  (0,0.2) to[out=90,in=90] (6,0.2) -- (5.9,0.2) to[out=90,in=90] (0.1,0.2) -- (0,0.2);
\draw[shift={(-3,0)}] (3,2)
  node[above]{${\color{orange}\wp_{\overline{i-1}}}{\color{red}\wp_{\overline{i}}} \in \I$};
\shade[left color=red, right color=blue][dotted][line width=2pt][opacity=0.5]
  (0,0.2) to[out=90,in=90] (6,0.2) -- (5.9,0.2) to[out=90,in=90] (0.1,0.2) -- (0,0.2);
\draw (3,2) node[above]{${\color{red}\wp_{\overline{i}}}{\color{blue}\wp_{\overline{i+1}}} \in \I$};
\shade[left color=blue, right color=cyan][dotted][line width=2pt][opacity=0.5][shift={(3,0)}]
  (0,0.2) to[out=90,in=90] (6,0.2) -- (5.9,0.2) to[out=90,in=90] (0.1,0.2) -- (0,0.2);
\draw[shift={(3,0)}] (3,2)
  node[above]{${\color{blue}\wp_{\overline{i+1}}}{\color{cyan}\wp_{\overline{i+2}}} \in \I$};
\end{tikzpicture}
\caption{${\color{orange}\wp_{\overline{i-1}}}{\color{red}\wp_{\overline{i}}}$,
${\color{red}\wp_{\overline{i}}}{\color{blue}\wp_{\overline{i+1}}}$ and
${\color{blue}\wp_{\overline{i+1}}}{\color{cyan}\wp_{\overline{i+2}}}$ are generators of $\I$ }
\label{fig:per.path}
\end{figure}

\begin{proof}
(1) Assume that there is an arrow $\alpha$ with starting point {\rm(}resp.  ending point{\rm)} $s(\wp_{\overline{i}})$
such that $\wp_{\overline{i}} \alpha = 0$ {\rm(}resp.  $\alpha \wp_{\overline{i+1}} = 0${\rm)}.
Then $(\wp_{\overline{i}}, \wp_{\overline{i+1}})$ is not a perfect pair,
since $\alpha$ is not a subpath of $\wp_{\overline{i+1}}$ (resp.  $\wp_{\overline{i}}$).
We obtain a contradiction. Thus, we have
\begin{center}
$\wp_{\overline{i}} \alpha \ne 0$ {\rm(}resp.  $\alpha \wp_{\overline{i+1}} \ne 0${\rm)}.
\end{center}
Furthermore, by the definition of string algebra, we have $\ell(\wp_{\overline{i}}) = 1 = \ell(\wp_{\overline{i+1}})$.

(2) Assume there exists an integer $i$ such that $\wp_{\overline{i}}\wp_{\overline{i+1}}$ is not a generator of $\I$.
Since $\wp_{\overline{i}}\wp_{\overline{i+1}}=0$, it follows that
$\wp_{\overline{i}}\wp_{\overline{i+1}}$ has a non-trivial subpath $\wp$ which is a generator of $\I$. Then
\begin{itemize}
  \item $\wp_{\overline{i}}' = \wp_{\overline{i}} \sqcap \wp$ and $ \wp_{\overline{i+1}}' = \wp \sqcap \wp_{\overline{i+1}}$ are non-trivial path on $(\Q, \I)$;

  \item $\wp_{\overline{i}}' \sqcup \wp_{\overline{i+1}}' = \wp$ (see \Pic \ref{fig:non-per.path}).
\end{itemize}
\begin{figure}[H]
\begin{tikzpicture}
\shade[left color=orange!50, right color=orange][opacity=0.5][shift={(-3, 0)}] (1.5,0) ellipse (1.5cm and 0.25cm);
\draw[shift={(-3, 0)}][->] (0.2,0) node[left]{$\bullet$} -- (0.8,0) node[right]{$\bullet$};
\draw[shift={(-2, 0)}][->] (0.2,0) -- (0.8,0) node[right]{$\bullet$} [dash pattern=on 2pt off 2pt];
\draw[shift={(-1, 0)}][->] (0.2,0) -- (0.8,0) node[right]{$\bullet$};
\draw[shift={(-3, 0)}][orange] (1.5,-0.2) node[below]{$\wp_{\overline{i-1}}$};
\shade[left color=pink, right color=red][opacity=0.5] (1.5,0) ellipse (1.5cm and 0.25cm);
\draw[->] (0.2,0) -- (0.8,0) node[right]{$\bullet$};
\draw[shift={( 1, 0)}][->] (0.2,0) -- (0.8,0) node[right]{$\bullet$} [dash pattern=on 2pt off 2pt];
\draw[shift={( 2, 0)}][->] (0.2,0) -- (0.8,0) node[right]{$\bullet$};
\draw[red] (1.25,-0.2) node[below]{$\wp_{\overline{i}}$};
\shade[left color=blue, right color=blue!50][opacity=0.5][shift={( 3, 0)}] (1.5,0) ellipse (1.5cm and 0.25cm);
\draw[shift={( 3, 0)}][->] (0.2,0) -- (0.8,0) node[right]{$\bullet$};
\draw[shift={( 4, 0)}][->] (0.2,0) -- (0.8,0) node[right]{$\bullet$} [dash pattern=on 2pt off 2pt];
\draw[shift={( 5, 0)}][->] (0.2,0) -- (0.8,0) node[right]{$\bullet$};
\draw[shift={(3.5,0)}][blue] (1.5,-0.2) node[below]{$\wp_{\overline{i+1}}$};
\shade[left color=cyan, right color=cyan!50][opacity=0.5][shift={( 6, 0)}] (1.5,0) ellipse (1.5cm and 0.25cm);
\draw[shift={( 6, 0)}][->] (0.2,0) -- (0.8,0) node[right]{$\bullet$};
\draw[shift={( 7, 0)}][->] (0.2,0) -- (0.8,0) node[right]{$\bullet$} [dash pattern=on 2pt off 2pt];
\draw[shift={( 8, 0)}][->] (0.2,0) -- (0.8,0) node[right]{$\bullet$};
\draw[shift={( 6, 0)}][cyan] (1.5,-0.2) node[below]{$\wp_{\overline{i+2}}$};
\shade[left color=orange, right color=red][dotted][line width=2pt][opacity=0.5][shift={(-3,0)}]
  (0,0.2) to[out=90,in=90] (6,0.2) -- (5.9,0.2) to[out=90,in=90] (0.1,0.2) -- (0,0.2);
\draw[shift={(-3,0)}] (3,2)
  node[above]{${\color{orange}\wp_{\overline{i-1}}}{\color{red}\wp_{\overline{i}}} = 0$};
\shade[left color=red, right color=blue][dotted][line width=2pt][opacity=0.5]
  (0,0.2) to[out=90,in=90] (6,0.2) -- (5.9,0.2) to[out=90,in=90] (0.1,0.2) -- (0,0.2);
\draw (3,2) node[above]{${\color{red}\wp_{\overline{i}}}{\color{blue}\wp_{\overline{i+1}}} = 0$};
\shade[left color=blue, right color=cyan][dotted][line width=2pt][opacity=0.5][shift={(3,0)}]
  (0,0.2) to[out=90,in=90] (6,0.2) -- (5.9,0.2) to[out=90,in=90] (0.1,0.2) -- (0,0.2);
\draw[shift={(3,0)}] (3,2)
  node[above]{${\color{blue}\wp_{\overline{i+1}}}{\color{cyan}\wp_{\overline{i+2}}} = 0$};
\draw[purple][line width=1pt][opacity=0.5] (3,0) ellipse (1.5cm and 1.5cm);
\draw[purple][->] (3,-1.6) -- (3,-2) node[below]{$\wp$ is a generator of $\I$};
\shade[left color=green!50, right color=green][opacity=0.5] (2.25,0) ellipse (0.75cm and 0.75cm);
\draw[green] (2.25,-1) node{$\wp_{\overline{i}}'$};
\shade[left color=green!50, right color=green][opacity=0.5] (3.75,0) ellipse (0.75cm and 0.75cm);
\draw[green] (3.75,-1) node{$\wp_{\overline{i+1}}'$};
\end{tikzpicture}
\caption{$\wp_{\overline{i}}\wp_{\overline{i+1}}$ has a non-trivial subpath $\wp$ which is a generator of $\I$ }
\label{fig:non-per.path}
\end{figure}
Thus,
\[\wp_{\overline{i}}' \wp_{\overline{i+1}} = 0 \ \text{and}\
\wp_{\overline{i}} \wp_{\overline{i+1}}' = 0, \]
which yields a contradiction since that $\wp_{\overline{i}}'$ and $\wp_{\overline{i+1}}'$ are non-trivial subpath of $\wp_{\overline{i}}$ and $\wp_{\overline{i+1}}$ respectively.
\end{proof}

Let $(p=\wp_1, \wp_2, \cdots, \wp_t=p)$ be a PPS on a relation cycle $\C$ of a string algebra $A=\kk\Q/\I$.
By Lemma \ref{lemm:per.path}, all paths of the form $\wp_{\overline{i}}\wp_{\overline{i+1}}$ are generators of $\I$ which are paths on $\C$.
Thus, we obtain a sequence of zero relations $(\wp_1\wp_2, \wp_2\wp_3, \cdots, \wp_{t-1}\wp_t, \wp_t\wp_1)$
and call it a {\defines perfect relation sequence} (=PRS) corresponding to the above PPS.
Furthermore, the following proposition holds directly.

\begin{proposition} \label{prop:perpaths}
A sequence of generators $(r_0, r_1, \cdots, r_{t-1})$ of $\I$ is a PRS if and only if the following statements hold.
\begin{itemize}
  \item $r_{\overline{i}} = (r_{\overline{i-1}} \sqcap r_{\overline{i}}) \cdot (r_{\overline{i}} \sqcap r_{\overline{i+1}})$.
  \item For any $i\in \mathbb{N}$, if there is an arrow $\alpha$ {\rm(}resp.  $\beta${\rm)}
    such that $t(\alpha)=s(\wp_{\overline{i}})$ {\rm(}resp.  $s(\beta)=t(\wp_{\overline{i-1}})${\rm)},
    then $\alpha\wp_{\overline{i}}\ne 0$ {\rm(}resp.  $\beta\wp_{\overline{i-1}} \ne 0${\rm)},
    where $\wp_{\overline{i}} = r_{\overline{i}}\sqcap r_{\overline{i+1}}$ for all $i$.
\end{itemize}
\end{proposition}

\begin{construction}\rm \label{const:perpaths}
By Proposition \ref{prop:perpaths}, we can use PRSs to find all PPSs.
\begin{itemize}
  \item[Step] 1.
  Find a relation cycle $\C$ of $(\Q, \I)$.

  \item[Step] 2.
  If there is a PRS $(r_0, r_1, \cdots, r_{t-1})$ on $\C$,
  then $(p=\wp_1, \wp_2, \cdots, \wp_n=p)$ is a PPS by Lemma \ref{lemm:per.path},
  where $\wp_{\overline{i}}$ is the path $r_{\overline{i}}\sqcap r_{\overline{i+1}}$ on $\C$ ($0\le i\le t-1$).
  Otherwise, $\C$ not provides PPS.
\end{itemize}
\end{construction}

\begin{example} \rm \label{ex:1}
Let $A=\kk\Q/\I$ be a string algebra with $\Q=$
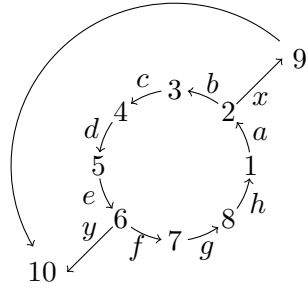
\begin{center}
\begin{tikzpicture}[scale=0.5]
\draw[->][rotate=  0+10] (2,0) arc (0:25:2);
\draw[->][rotate= 45+10] (2,0) arc (0:25:2);
\draw[->][rotate= 90+10] (2,0) arc (0:25:2);
\draw[->][rotate=135+10] (2,0) arc (0:25:2);
\draw[->][rotate=180+10] (2,0) arc (0:25:2);
\draw[->][rotate=225+10] (2,0) arc (0:25:2);
\draw[->][rotate=270+10] (2,0) arc (0:25:2);
\draw[->][rotate=315+10] (2,0) arc (0:25:2);
\draw[rotate=  0] (2,0) node{$1$}; \draw[rotate=  0] (2.25,0.85) node{$a$};
\draw[rotate= 45] (2,0) node{$2$}; \draw[rotate= 45] (2.25,0.85) node{$b$};
\draw[rotate= 90] (2,0) node{$3$}; \draw[rotate= 90] (2.25,0.85) node{$c$};
\draw[rotate=135] (2,0) node{$4$}; \draw[rotate=135] (2.25,0.85) node{$d$};
\draw[rotate=180] (2,0) node{$5$}; \draw[rotate=180] (2.25,0.85) node{$e$};
\draw[rotate=225] (2,0) node{$6$}; \draw[rotate=225] (2.25,0.85) node{$f$};
\draw[rotate=270] (2,0) node{$7$}; \draw[rotate=270] (2.25,0.85) node{$g$};
\draw[rotate=315] (2,0) node{$8$}; \draw[rotate=315] (2.25,0.85) node{$h$};
\draw[rotate= 45][->] (2.3,0) -- (4,0) node[right]{$9$};
\draw[rotate= 45] (2.5,0) node[right]{$x$};
\draw[rotate= 45+180][->] (2.3,0) -- (4,0) node[left]{$10$};
\draw[rotate= 45+180] (2.5,0) node[left]{$y$};
\draw[rotate= 45+5][->] (4.3,0) arc (0:160:4.3);
\end{tikzpicture}
\end{center}
and $\I = \langle abcd, cdef, efgh, ghab, habc, defg, ax, ey \rangle$.
There are eight PRSs as follows:
\begin{align}
  & s_{11} = (abcd, defg, efgh, habc, abcd), && s_{12} = (defg, efgh, habc, abcd, defg), \nonumber \\
  & s_{13} = (efgh, habc, abcd, defg, efgh), && s_{14} = (habc, abcd, defg, efgh, habc); \nonumber \\
  & s_{21} = (abcd, cdef, efgh, ghab, abcd), && s_{22} = (cdef, efgh, ghab, abcd, cdef), \nonumber \\
  & s_{23} = (efgh, ghab, abcd, cdef, efgh), && s_{24} = (ghab, abcd, cdef, efgh, ghab). \nonumber
\end{align}
Thus, we obtain the following corresponding between perfect paths and PPSs:
\begin{align}
  & p_{11} = abc \leftrightarrow (abc, d, efg, h, abc), && p_{12} = d \leftrightarrow (d, efg, h, abc, d), \nonumber \\
  & p_{13} = efg \leftrightarrow (efg, h, abc, d, efg), && p_{14} = h \leftrightarrow (h, abc, d, efg, h); \nonumber \\
  & p_{21} = ab \leftrightarrow (ab, cd, ef, gh, ab), && p_{22} = cd \leftrightarrow (cd, ef, gh, ab, cd), \nonumber \\
  & p_{23} = ef \leftrightarrow (ef, gh, ab, cd, ef), && p_{24} = gh \leftrightarrow (gh, ab, cd, ef, gh). \nonumber
\end{align}
\end{example}

\subsection{Gorenstein projective modules} \label{subsect:Gproj}

Let $A$ be a finite-dimensional $\kk$-algebra. A {\defines Gorenstein-projective module} ($=$G-projective module) $G$ in $\modcat A$ is a module which has a $\Hom_{A}(-, A)$-exact complete projective resolution,
that is, there is an exact sequence of projective modules
\[\cdots \longrightarrow P^{-2}
\mathop{\longrightarrow}\limits^{d^{-2}} P^{-1} \mathop{\longrightarrow}\limits^{d^{-1}} P^{0}
\mathop{\longrightarrow}\limits^{d^{0}}  P^{1}  \mathop{\longrightarrow}\limits^{d^{1}}
P^{2} \longrightarrow \cdots \]
in $\modcat A$ which remains exact after applying the functor $\Hom_A(-,A)$, such that $G \cong \mathrm{Im}d^{-1}$ \cite{AB1969, EJ1995}.
We denote $\Gproj(A)$ by the fully subcategory of $\modcat A$ whose objects are G-projective.
Obviously, any projective module are G-projective. We say a G-projective is {\defines non-trivial} if it is not projective.

The following theorem provide a description of NTIG-projective modules.

\begin{theorem}[{\cite[Theorem 4.1]{C2018}}] \label{thm-Chen}
Let $A=\kk\Q/\I$ be a monomial algebra. Then there is a bijection
\[ \perpath (\Q,\I) \to \NTIGproj(A), p \mapsto pA \]
from the set of all perfect paths to that of all NTIG-projective modules.
\end{theorem}

\begin{example} \rm \label{ex:1-1}
Let $A=\kk\Q/\I$ be a string algebra given in Example \ref{ex:1}.
All NTIG-projective modules are $abcA$ $\cong (4)$, $dA$ $\cong \left(\begin{smallmatrix} 5 \\ 6 \\ 7 \end{smallmatrix}\right)$,
$efgA$ $\cong (8)$, $hA$  $\cong \left(\begin{smallmatrix} 1 \\ 2 \\ 3 \end{smallmatrix}\right)$,
$abA$ $\cong \left({^3_4}\right)$, $cdA$ $\cong \left({^1_2}\right)$, $efA$ $\cong \left({^7_8}\right)$, $cdA$ $\cong \left({^5_6}\right)$.
\end{example}

\subsection{The irreducible morphisms in $\Gproj(A)$}
The present subsection is devoted to describing the irreducible morphisms in $\Gproj(A)$.

First of all, we recall the description of irreducible morphisms between string modules over string algebra \cite{BR1987}.
Let $w$ be a string. The irreducible morphisms starting at $\M(w)$ can be described by modifying the string $w$ on the left or on the right:
\begin{itemize}
  \item[Case] 1. If there is an arrow $a\in \Q_1$ such that $aw$ is a string,
    then let $v$ be the maximal direct string starting at $s(a)$, and define $w_l$ to be the string $v^{-1}aw$.

  \item[Case] 2. If there is no arrow $a\in \Q_1$ such that $aw$ is a string,
    Then, either $w$ is a direct string or we can write $w=va^{-1}w'$, where $a\in \Q_1$, $w'$ is a string, and $v$ is a maximal direct substring starting at $s(w)$.
    In the former case, we define $w_l$ to be the trivial string, and in the latter case $w_l:=w'$.
\end{itemize}
The strings $v^{-1}a$ and $va^{-1}$ are called {\defines hook} and {\defines cohook}, respectively (see \Pic \ref{fig:left changing string} (L1) and (L2), respectively).
Dually, we can define hook and cohook in the case shown in \Pic \ref{fig:right changing string}.

\begin{figure}[H]
\centering
\begin{tikzpicture}
\draw (-3,0.8) node{$w_l=$};
\draw (0,0) node{$
\xymatrix@R=0.35cm{
& \bullet \ar@{~>}[ldd]^v \ar[rd]^a &  & &  \\
& & \bullet \ar@{~}[rr]^{w} & & \bullet \\
\bullet & &  &  &
}
$};
\draw(0,-1.3) node{(L1)};
\end{tikzpicture}
\ \ \ \
\begin{tikzpicture}
\draw (-3,0.8) node{$w=$};
\draw (0,0) node{$
\xymatrix@R=0.35cm{
\bullet \ar@{~>}[rdd]_v &   &  & &  \\
& & \bullet \ar[ld]^a \ar@{~}[rr]_{w_l} & & \bullet  \\
 & \bullet & & &
}
$};
\draw(0,-1.3) node{(L2)};
\end{tikzpicture}
\caption{Left: adding a hook on $s(w)$; Right: removing a cohook from $s(w)$}
\label{fig:left changing string}
\end{figure}
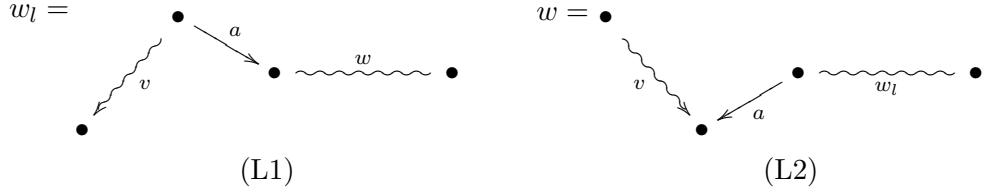

The definition of $w_r$ is dual (see \Pic \ref{fig:right changing string} (R1) and (R2), respectively).

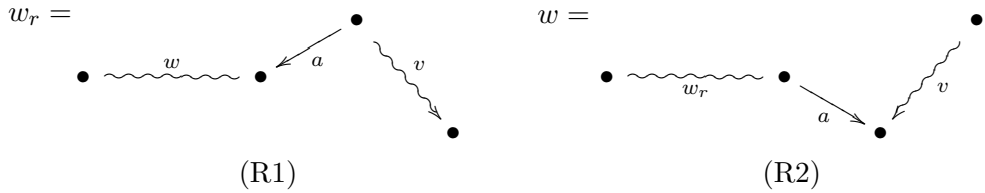
\begin{figure}[H]
\centering
\begin{tikzpicture}
\draw (-3,0.8) node{$w_r=$};
\draw (0,0) node{$
\xymatrix@R=0.35cm{
  &  &  & \bullet\ar[ld]^a \ar@{~>}[rdd]^v &  \\
\bullet \ar@{~}[rr]^{w}  &  & \bullet &  &  \\
  &  &  &  &  \bullet
}
$};
\draw(0,-1.3) node{(R1)};
\end{tikzpicture}
\ \ \ \
\begin{tikzpicture}
\draw (-3,0.8) node{$w=$};
\draw (0,0) node{$
\xymatrix@R=0.35cm{
&  &  &  & \bullet \ar@{~>}[ldd]^v\\
\bullet \ar@{~}[rr]_{w_r}&  & \bullet \ar[rd]_a &  &  \\
&  &  & \bullet  &
}
$};
\draw(0,-1.3) node{(R2)};
\end{tikzpicture}
\caption{Left: adding a hook on $t(w)$; Right: removing a cohook from $t(w)$}
\label{fig:right changing string}
\end{figure}

It is well-kwon that any irreducible morphism between two string modules is one of the form of
\begin{itemize}
\item[(L1)] the monomorphism $\M(w) \hookrightarrow \M(w_l)$;
\item[(R1)] the monomorphism $\M(w) \hookrightarrow \M(w_r)$;
\item[(L2)] the epimorphism $\M(w_l) \twoheadrightarrow \M(w)$;
\item[(R2)] the epimorphism $\M(w_r) \twoheadrightarrow \M(w)$.
\end{itemize}

Next, we provide some lemmas to describe irreducible morphisms in $\Gproj(A)$.

\begin{lemma}\label{lemm:Gproj(I)}
Let $A = \kk\Q/\I$ be a string algebra and $G\cong pA$ be an NTIG-projective module, where $p$ is a path given by a PPS $(p=\wp_1, \wp_2, \cdots, \wp_t=p)$ on a relation cycle $\C$.
Then the homomorphism $\hbar_p: P(t(p))\to P(s(p))$ corresponding to the path $p$ has a decomposition as follows
\[ \xymatrix{
P(t(p)) \ar[rd]_{f_2} \ar[rr]^{\hbar_p} & & P(s(p)), \\
& pA \ar@{->}[ru]_{f_1} &
} \]
where $f_1$ {\rm(}resp.  $f_2$ {\rm)} can not be decomposed through any indecomposable projective module.
\end{lemma}

\begin{proof}
Assume that $\wp_i = \alpha_{i,1}\alpha_{i,2}\cdots \alpha_{i,l_i}$ ($1\le i\le t-1$), $t(\alpha_{i,j}) = v_{i,j}$ ($1\le j\le l_i$) and $e_{i,j}$ be the idempotent corresponding to $v_{i,j}$.
Then the string corresponding to $P(t(p))$ equals to
\[ \M^{-1}(e_{1,l_1}A) = \hat{\wp}^{-1}\wp_2\alpha_{3,1}\alpha_{3,2}\cdots \alpha_{3,l_3-1}
= (\hat{\wp}^{-1}\wp_2\alpha_{3,1}\alpha_{3,2}\cdots \alpha_{3,l_3-2}) \cdot \alpha_{3,l_3-1} {e_{3,l_3-1}}\]
which can be seen as a string $w_1$, where $\hat{\wp}$ is a path starting at the point $t(p)=t(\wp_1)=s(\wp_2)$.
\begin{itemize}
\item
If $\ell(\wp_2)=\ell(\wp_3)=1$ (in this case, $l_2=l_3=1$ and $w_1 = \hat{\wp}^{-1}\alpha_{2,1}$)
and there exists an arrow ($\ne\alpha_{2,1}$) ending at $v_{2,l_2}$,
directly deleting the last arrow $\alpha_{2,1} = \wp_2 $ of $w_1$, we obtain that
$w_2 = \hat{\wp}^{-1}$ satisfies $\M(w_2) \cong \M(\hat{\wp}) \cong pA$ (notice that if $\ell(p)>1$ then $\ell(\hat{\wp}) = 0$ by the definition of perfect pair),
and
 the homomorphism $\hbar_{\wp_2}$ induced by $\wp_2=\alpha_{2,1}$ is not irreducible,
since $\hbar_{\wp_2}$ has a canonical decomposition.

\item
Otherwise, there exists no arrow $\beta_{3,j}$ ($\ne \alpha_{3,j}$) ending at $v_{3,j}$ such that $\beta_{3,j}\alpha_{3,{j+1}} \ne 0$ ($1\le j\le l_3-1$).
Then, by (R2), we have
\begin{center}
$(w_1)_r = \hat{\wp}^{-1}\wp_2\alpha_{3,1}\alpha_{3,2}\cdots \alpha_{3,l_3-2}$ ($=:w_2$),
\end{center}
and obtain an irreducible homomorphism $\hbar_{\alpha_{3,l_3-2}} \in \Hom_A(\M(w_1), \M(w_2))$ which is induced by the path $\alpha_{3,l_3-2}$ of length one.
\end{itemize}

Thus, for any $1\le j< l_3-1$, we have
\begin{center}
$w_j = \hat{\wp}^{-1}\wp_2\alpha_{3,1}\alpha_{3,2}\cdots \alpha_{3,l_3-j}
= (\hat{\wp}^{-1}\wp_2\alpha_{3,1}\alpha_{3,2}\cdots \alpha_{3,l_3-(j+1)})\alpha_{3,l_3-j}e_{3,l_3-j} = $
\end{center}
\begin{center}
$\xymatrix@R=0.5cm{
& &  & & v_{3,l_3-j} \ar[ldd]^{\text{trivial path}\ e_{3,l_3-j}, } \\
\bullet \ar@{~}[rr]^{(w_j)_r = \hat{\wp}^{-1}\wp_2\alpha_{3,1}\cdots\alpha_{3,l_3-(j+1)}}
& & \bullet \ar[rd]_{\alpha_{3,l_3-j}} & & \\
& & & v_{3,l_3-j} &
}$
\end{center}
and $w_{j+1} = \hat{\wp}^{-1}\wp_2\alpha_{3,1}\cdots\alpha_{3,l_3-(j+1)}$ equals either $(w_{j})_r$ or the string obtained by deleting the last arrow of $w_j$.
It is easy to see that $\M(w_{l_3-1}) \cong pA$.

On the other hand, assume $\hat{\wp}=\beta_1\cdots\beta_s$. Then $w_{l_3-1}$ is equivalent to
\[(\alpha_{2,l_2-1}^{-1}\cdots\alpha_{2,1}^{-1}\beta_1\cdots\beta_{s-1})\beta_{s} e_{t(\beta_s)} =: \hat{w}_1. \]
Similarly, we have $\hat{w}_j = \alpha_{2,l_2-1}^{-1} \cdots \alpha_{2,1}^{-1}\beta_1 \cdots \beta_{s-(j-1)}$ is the string such that $\hbar_{\beta_{s-(j-1)}}: \M(\hat{w}_j) \to \M(\hat{w}_{j+1})$
is either an irreducible homomorphism in $\modcat A$ or a epimorphism which is not decomposed by any indecomposable projective module.
We obtain a homomorphism
\[\hbar_{\hat{\wp}}: \M(\hat{w}_1) \to \M(\hat{w}_{s+1}), \]
where $\hat{w}_{s+1}$ is equivalent to
\[ \hat{w}_{s+1}^{-1} = \alpha_{2,1} \alpha_{2,2} \cdots \alpha_{2,l_2-1} =: \omega_1. \]
Then there are two situations as follows:
\begin{itemize}
  \item $\ell(\wp_{t-1})=\ell(\wp_1) = 1$.
    In this case, $\wp_1$ is an arrow, we have $(\omega_1)_l = \tilde{\wp}^{-1}\wp_1\omega_1$,
    and $\M((\omega_1)_l) \cong e_{s(\wp_1)}A = P(s(p))$.

  \item Otherwise, by the definition of string algebra,
    for any arrow $\gamma_{1,j} (\ne \alpha_{1,j})$ ending (resp.  starting) to $v_{1,j}$ (resp.  $v_{1,j-1}$),
    where $1\le j\le l_1-1$ (resp.  $1+1\le j\le l_1$), we always have $\gamma_{1,j}\alpha_{1,j+1} = 0$ (resp.  $\alpha_{1,j-1}\gamma_{1,j} = 0$).
    Thus, we can apply (L1) to $\omega_1$ and obtain the string
    \begin{center}
      $\omega_2 = \alpha_{1,l_1}\omega_1 = \alpha_{1,l_1}\cdot (\alpha_{2,1} \cdots \alpha_{2,l_2-1})$.
    \end{center}
\end{itemize}

Repeating the above steps, we obtain a morphism from $M$  to $pA$, where
    \begin{center}
      $M \cong \M\big((\alpha_{2,l_2-1}^{-1} \cdots \alpha_{2,1}^{-1}) \cdot \wp_1^{-1}\big)$.
    \end{center}
Notice that $M$ is a direct summand of $\rad P(t(a))$. Thus, there is a homorphism from $pA$ to $P(t(p))$.

    Therefore, $\hbar_p$ is not irreducible in $\Gproj(A)$ and has the following decomposition:
\begin{align}\label{diagram in lemma:1-Gproj}
\xymatrix{
P(t(p)) \ar[rd]_{f_2 = \tilde{\wp}p\hat{\wp}\cdot (?)} \ar[rr]^{\hbar_p} & & P(s(p)). \\
& pA \ar@{^(->}[ru]^{\le}_{f_1 = p\cdot (?)} &
}
\end{align}

Next, we show that $f_2=\tilde{\wp}p\hat{\wp}\cdot (?)$ and $f_1=p\cdot (?)$ are not decomposed through any indecomposable projective module not isomorphic to $P(t(p))$ and $P(s(p))$, respectively.
Assume that there is an indecomposable projective module $P \not\cong P(s(p))$ and two homomorphisms $g_2: pA \to P$ and $g_1: P \to P(s(p))$ such that $f_1 = g_1g_2$.
Consider the canonical decomposition $g_2=g_{21}g_{22}$ of $g_2$:
\[\xymatrix{ pA \ar[r]_{g_{22}} \ar@/^1pc/[rr]^{g_2} & g_{22}(pA) \ar[r]_{g_{21}} & P, }\]
we have two cases as follows which respectively yield contradictions.
\begin{itemize}
\item[(1)]
If $\hat{\wp} = \gamma_1\cdots\gamma_l$ is a path of length $l\ge 1$,
then by Proposition \ref{prop:perpaths}, $\alpha_{1,\ell_1} \gamma_1 \ne 0$.
Thus, by the definition of string algebra, we have $\ell(\wp_1) = 1 = \ell(\wp_2)$  and $\ell_1=1$.
In this case,
\begin{center}
$\M^{-1}(pA) = \hat{\wp}$, 
$\M^{-1}(P(s(p))) = \tilde{\wp}^{-1} \alpha_{1,1}\gamma_1\cdots\gamma_l$,

and $\M^{-1}(g_{22}(pA)) = \gamma_1\cdots\gamma_{l'}$, where $1\le l'\le l$.
\end{center}
Since $g_{21}(pA)$ is a submodule of $P$ and $A$ is a string algebra, there is a path $q$ from $v\in Q_0$ to $s(p)$ such that
\begin{center}
$\M^{-1}(P) = \M^{-1}(P(v))  = \hat{q}^{-1}q\alpha_{1,1}\gamma_1\cdots\gamma_{l'} \ne 0$
\end{center}
for some path $\hat{q}$ starting at $v$, as shown in \Pic \ref{fig-1: pf in lemm:Gproj(I)}, if there exists $\gamma_{l'+1}$ starting at $t(\gamma_{l'})$ then the path $qp\gamma_1\cdots\gamma_{l'+1}$ is a generator in $\I$.
It is easy to see that any quotient of $P$ is not a submodule of $P(s(p))$ by the canonical decomposition of $g_1$, this is a contradiction.

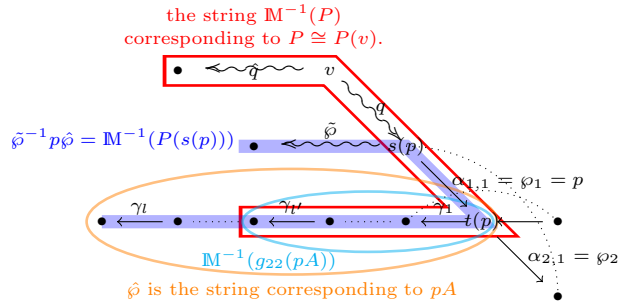
\begin{figure}[H]
\centering
\begin{tikzpicture}[scale=1] \tiny
\draw[red][opacity=1.0] [line width=12pt] (-4.2,2) -- (-2,2) -- (0,0) -- (-3.2,0);
\draw[white] [line width=10pt] (-4.15,2) -- (-2,2) -- (0,0) -- (-3.15,0);
\draw[red] (-3,2.5) node[above]{the string $\M^{-1}(P)$};
\draw[red] (-3,2.2) node[above]{corresponding to $P\cong P(v)$.};
\draw[blue][opacity=0.3][line width=5pt] (-3.2,1) -- (-1,1) -- (0,0) -- (-5,0);
\draw[blue] (-3.15,1.1) node[left]{$\tilde{\wp}^{-1}p\hat{\wp} = \M^{-1}(P(s(p)))$};
\draw[orange][opacity=0.5][line width=1pt] (-2.5,0) ellipse (2.7cm and 0.7cm);
\draw[orange] (-2.5,-0.7) node[below]{$\hat{\wp}$ is the string corresponding to $pA$};
\draw[cyan][opacity=0.5][line width=1pt] (-1.5,0) ellipse (1.65cm and 0.4cm);
\draw[cyan] (-2.8,-0.26) node[below]{$\M^{-1}(g_{22}(pA))$};
\draw (0,0) node{$t(p)$};
\draw[->] (0.2,-0.2) -- (0.8,-0.8); \draw (1,-1) node{$\bullet$};
\draw (-0.5,0.5) node[right]{$\alpha_{1,1}=\wp_1=p$};
\draw[shift={(-1,1)}][->] (0.2,-0.2) -- (0.8,-0.8);
\draw (-1,1) node{$s(p)$}; \draw (0.5,-0.5) node[right]{$\alpha_{2,1}=\wp_2$};
\draw[dotted] (-1,1) to[out=0,in=90] (1,-1);
\draw[shift={(1,0)}][->] (-0.2,0) -- (-0.8,0);
\draw[shift={(1,0)}] (0,0) node{$\bullet$};
\draw[->] (-0.2,0) -- (-0.8,0);
\draw (-1,0) node{$\bullet$};
\draw[dotted][shift={(-1,0)}] (-0.2,0) -- (-0.8,0);
\draw[shift={(-1,0)}] (-1,0) node{$\bullet$};
\draw[->][shift={(-2,0)}] (-0.2,0) -- (-0.8,0);
\draw[shift={(-2,0)}] (-1,0) node{$\bullet$};
\draw[dotted][shift={(-3,0)}] (-0.2,0) -- (-0.8,0);
\draw[shift={(-3,0)}] (-1,0) node{$\bullet$};
\draw[->][shift={(-4,0)}] (-0.2,0) -- (-0.8,0);
\draw[shift={(-4,0)}] (-1,0) node{$\bullet$};
\draw (-0.5,0) node[above]{$\gamma_1$};
\draw[shift={(-2,0)}] (-0.5,0) node[above]{$\gamma_{l'}$};
\draw[shift={(-4,0)}] (-0.5,0) node[above]{$\gamma_{l}$};
\draw[dotted] (1,0) to[out=135,in=45] (-1,0);
\draw[shift={(-2,2)}] (0,0) node{$v$};
\draw[shift={(-2,2)}] (0,0) node[left]{$\xymatrix@C=1.3cm{ & \ar@{~>}[l]}$};
\draw[shift={(-3,2)}] (0,0) node{$\hat{q}$};
\draw[shift={(-4,2)}] (0,0) node{$\bullet$};
\draw[shift={(-1,1)}] (-0.5,0.5) node[right]{$q$};
\draw[shift={(-1,1)}] (-0.5,0.5) node{$\xymatrix@C=0.8cm{ \ar@{~>}[rd] & \\ & }$};
\draw[shift={(-1,1)}] (0,0) node[left]{$\xymatrix@C=1.3cm{ & \ar@{~>}[l]}$};
\draw (-2,1) node[above]{$\tilde{\wp}$};
\draw[shift={(-3,1)}] (0,0) node{$\bullet$};
\end{tikzpicture}
\caption{The case that $\hat{\wp}$ is a path of length $l\ge 1$}
\label{fig-1: pf in lemm:Gproj(I)}
\end{figure}

\item[(2)] If $\hat{\wp}$ is trivial, then the lengths of $\wp_1$ and $\wp_2$ are greater than or equal to $1$.
    The case of $\ell(\wp_1) = \ell(\wp_2)=1$ is similar to case (1).
    We only consider the case such that at least one of $\ell(\wp_2) > 1$ and $\ell(\wp_1) > 1$ hold.
    Notice that there is no arrow $\alpha$ ending at $t(p)$.
    Otherwise, we have $\alpha\alpha_{2,1} \in \I$ by the definition of string algebra,
    then $(\wp_1, \wp_2)$ is not a perfect pair by Proposition \ref{prop:perpaths}.
    So, the string corresponding to $g_{22}(pA)$ is $\alpha_{2,1}\cdots\alpha_{2,l_2'}$, where $1\le l_2' \le l_2-1$.
    $g_{22}(pA)$ is a submodule of $P$, then there is a path $q$ form $v\in \Q_0$ to $s(p)$ such that
    \begin{center}
    $\M^{-1}(P(v)) = \M^{-1}(p) = \hat{q}^{-1}qp\alpha_{2,1}\cdots\alpha_{2,l_2'}$,
    \end{center}
    for some $\hat{q}$ starting at $v$, see \Pic \ref{fig-2: pf in lemm:Gproj(I)}
    (note that if there exists an arrow $\delta_{1,i}$ ($\ne \alpha_{1,i}$, $1\le i\le \ell_1-1$) with ending point $t(\delta_{1,i}) = v_{1,i}$, then $\delta_{1,i}\alpha_{1,i+1} \in \I$).
    It is easy to see that any quotient of $P$ is not a submodule of $P(s(p))$, this is a contradiction.

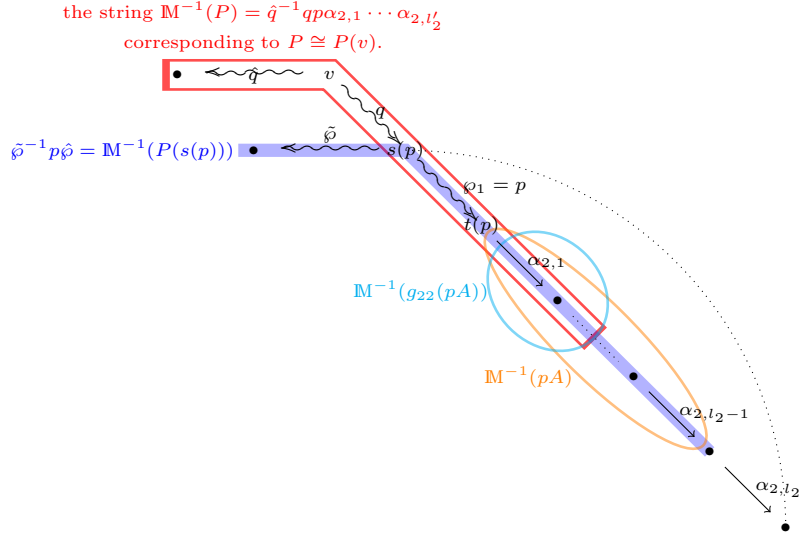
\begin{figure}[htbp]
\centering
\begin{tikzpicture}[scale=1] \tiny
\draw[red][opacity=0.7] [line width=12pt] (-4.2,2) -- (-2,2) -- (0,0) -- (1.5,-1.5);
\draw[white] [line width=10pt] (-4.1,2) -- (-2,2) -- (0,0) -- (1.45,-1.45);
\draw[red] (-3,2.5) node[above]{the string $\M^{-1}(P)
  = \hat{q}^{-1}qp\alpha_{2,1}\cdots\alpha_{2,l_2'} $};
\draw[red] (-3,2.2) node[above]{corresponding to $P\cong P(v)$.};
\draw[blue][opacity=0.3][line width=5pt] (-3.2,1) -- (-1,1) -- (0,0) -- (3,-3);
\draw[blue] (-3.15,1) node[left]{$\tilde{\wp}^{-1}p\hat{\wp} = \M^{-1}(P(s(p)))$};
\draw[orange][opacity=0.5][line width=1pt][rotate around={315:(1.5,-1.5)}] (1.5,-1.5) ellipse (2cm and 0.5cm);
\draw[orange] (1.3,-2) node[left]{$\M^{-1}(pA)$};
\draw[cyan][opacity=0.5][line width=1pt][rotate around={315:(0.87,-0.87)}] (0.87,-0.87) ellipse (0.87cm and 0.7cm);
\draw[cyan] (0.19,-0.87) node[left]{$\M^{-1}(g_{22}(pA))$};
\draw (0,0) node{$t(p)$};
\draw[->] (0.2,-0.2) -- (0.8,-0.8); \draw (1,-1) node{$\bullet$};
\draw[shift={(1,-1)}][dotted] (0.2,-0.2) -- (0.8,-0.8); \draw[shift={(1,-1)}] (1,-1) node{$\bullet$};
\draw[shift={(2,-2)}][->] (0.2,-0.2) -- (0.8,-0.8); \draw[shift={(2,-2)}] (1,-1) node{$\bullet$};
\draw[shift={(3,-3)}][->] (0.2,-0.2) -- (0.8,-0.8); \draw[shift={(3,-3)}] (1,-1) node{$\bullet$};
\draw (-0.35,0.5) node[right]{$\wp_1=p$};
\draw[shift={(-0.5,0.5)}] (0,0) node{$\xymatrix@C=0.8cm{ \ar@{~>}[rd] & \\ & }$};
\draw (-1,1) node{$s(p)$}; \draw (0.5,-0.5) node[right]{$\alpha_{2,1}$};
\draw[shift={(2,-2)}] (0.5,-0.5) node[right]{$\alpha_{2,l_2-1}$};
\draw[shift={(3,-3)}] (0.5,-0.5) node[right]{$\alpha_{2,l_2}$};
\draw[dotted] (-1,1) to[out=0,in=90] (4,-4);
\draw[shift={(-2,2)}] (0,0) node{$v$};
\draw[shift={(-2,2)}] (0,0) node[left]{$\xymatrix@C=1.3cm{ & \ar@{~>}[l]}$};
\draw[shift={(-3,2)}] (0,0) node{$\hat{q}$};
\draw[shift={(-4,2)}] (0,0) node{$\bullet$};
\draw[shift={(-1,1)}] (-0.5,0.5) node[right]{$q$};
\draw[shift={(-1,1)}] (-0.5,0.5) node{$\xymatrix@C=0.8cm{ \ar@{~>}[rd] & \\ & }$};
\draw[shift={(-1,1)}] (0,0) node[left]{$\xymatrix@C=1.3cm{ & \ar@{~>}[l]}$};
\draw (-2,1) node[above]{$\tilde{\wp}$};
\draw[shift={(-3,1)}] (0,0) node{$\bullet$};
\end{tikzpicture}
\caption{The case that $\hat{\wp}$ is a trivial path }
\label{fig-2: pf in lemm:Gproj(I)}
\end{figure}
\end{itemize}

Similarly, we can show that $f_2=\tilde{\wp}p\hat{\wp}\cdot (?)$ can not decomposed through any indecomposable projective module not isomorphic to $P(t(p))$.
\end{proof}

\begin{lemma} \label{lemm:pA-left}
Let $A=\kk\Q/\I$ be a string algebra contains at least one relation cycle $\C=\alpha_1\cdots\alpha_{\ell}$ of length $\ell$,
and $p_{[i, i+j]}$ be the path $\alpha_i\alpha_{i+1} \cdots \alpha_{i+j-1}$ {\color{blue}$\ne 0$} of length $j$, where the indices are taken modulo $l$ if necessary.
Suppose that $p_{[1,l]}A$ and $p_{[1,n]}A$ are two NTIG-projective modules over $A$ $(l<n)$,
then there exists a homomorphism $\hbar_{p_{[l,n]}}: p_{[1,n]}A \to p_{[1,l]}A$ induced by the path $p_{[l,n]}$ on $\C$.
Furthermore, if $p_{[1, m]}A$ is not G-projective for all $l<m<n$,
then \[\dim_{\kk}\Irr_{\Gproj(A)}(p_{[1,n]}A, p_{[1,l]}A) = 1. \]
\end{lemma}

\begin{proof}
Assume that there is an integer $m$  ($l<m<n$) such that $p_{[1, m]}A$ is a NTIG-projective module.
Let $s(\alpha_i)=i$ for all $1\le i\le n-1$.
Notice that
\begin{center}
$\M^{-1}(p_{[1,n]}A) = p_n^{-1} \cdot (\alpha_n\cdots\alpha_N) \notin \projcat(A)$

for some path $p_n$ starting at $n$ and $N\ge n$,
\end{center}
then we have
\begin{center}
$\M^{-1}(p_{[1,m]}A) = p_m^{-1} \cdot (\alpha_m\cdots\alpha_{n-1}) \cdot (\alpha_n\cdots\alpha_{N}) \notin \projcat (A)$

for some path $p_m$ starting at $m$,
\end{center}
and
\begin{center}
$\M^{-1}(p_{[1,l]}A) = p_l^{-1} \cdot (\alpha_l\cdots\alpha_{m-1}) \cdot (\alpha_m\cdots\alpha_{n-1}) \cdot (\alpha_n\cdots\alpha_{N}) \notin \projcat(A)$

for some path $p_l$ starting at $l$,
\end{center}
where $p_k = \beta_{k,1}\cdots\beta_{k,u_k}$ ($1\le k\le n$) a path of length $u_k$ starting at the
point $k\in\Q_0$.

Similar to the proof of Lemma \ref{lemm:Gproj(I)}, we obtain three homomorphisms
\begin{align}
  & \hbar_{[m,n]} \in \Hom_A(p_{[1,n]}A, p_{[1,m]}A), \nonumber \\
  & \hbar_{[l,m]}\in \Hom_A(p_{[1,m]}A, p_{[1,l]}A),  \nonumber \\ \ \text{and}\ \ \ \
  & \hbar_{[l,n]}\in \Hom_A(p_{[1,n]}A, p_{[1,l]}A) \nonumber
\end{align}
which are induced by paths $p_{[m,n]}$, $p_{[l,m]}$ and $p_{[l,n]}$, respectively.
It is easy to see that $\hbar_{[l,n]} = \hbar_{[l,m]}\hbar_{[m,n]}$, that is, $p_{[1,m]}A$ is a NTIG-projective module if and only if $\hbar_{[l,n]}$ is not irreducible in $\Gproj(A)$.

Next, we prove $\dim_{\kk}\Irr_{\Gproj(A)}(p_{[1,n]}A, p_{[1,l]}A) = 1$.
We know that all morphisms in $\Irr_{\Gproj(A)}(p_{[1,n]}A, p_{[1,l]}A)$ are induced by some paths on $(\Q, \I)$ whose starting points are $l$ and ending points are $n$.
Assume that there is a path $p' = \beta_1\cdots\beta_s$ ($s\ge 1$) with $s(p')=l$ and $t(p')=n$ such that $p' \ne p_{[l,n]}$.
Note that $l(p_{[1,n]}) > l(p_{1,l}) \ge 1$, we have $\beta_s\alpha_{n+1} = 0$ by the definition of string algebras.
Since $p_{[1,n]}A$ is a NTIG-projective module, there exists a path $p''$ starting at $t(p_{[1,n]} A)$ on $\C$ such that $(p_{[1,n]}, p'')$ is a perfect pair.
Thus, we obtain that $\beta_s p'' = 0$ by $\beta_s\alpha_{n+1} = 0$.
By the definition of perfect pair, $p_{[1,n]}$ is a subpath of $\beta_s$. We obtain a contradiction.
Thus there exists a unique path from the vertex $l$ to the vertex $n$, and then the statement follows.

\end{proof}

Dually, we have the following Lemma.

\begin{lemma} \label{lemm:pA-right}
Keep the notations from Lemma \ref{lemm:pA-left}.
If there are two NTIG-projective modules $p_{[l,\ell]}A$ and $p_{[n,\ell]}A$ over $A$ {\rm(}$l>n${\rm)},
then there exists a homomorphism $\hbar_{p_{[n,l]}}: p_{[l,\ell]}A \to p_{[n,\ell]}A$ induced by the path $p_{[n,l]}$ on $\C$.
Furthermore, if $p_{[m, \ell]}A$ is not G-projective for all $n<m<l$,
then \[\dim_{\kk}\Irr_{\Gproj(A)}(p_{[l,\ell]}A, p_{[n,\ell]}A) = 1. \]
\end{lemma}

By Lemmas \ref{lemm:Gproj(I)}, \ref{lemm:pA-left} and \ref{lemm:pA-right}, we obtain the following lemma.

\begin{lemma}\label{lemm:Gproj-length=1}
Keep the notations in Lemma \ref{lemm:Gproj(I)}. If $\ell(p)=1$, then
\begin{align}\label{formula in prop:Gproj-length=1}
\dim_{\kk}\Irr_{\Gproj(A)}(P(t(p)), G) = 1 = \dim_{\kk}\Irr_{\Gproj(A)}(G, P(s(p))).
\end{align}
\end{lemma}

A finite dimensional $\kk$-algebra $A$ is called {\defines CM-finite} if $\ind(\Gproj(A))$ contains at most a finite number of indecomposable G-projective modules (up to isomorphism).
$A$ is called {\defines CM-free} if $\Gproj(A) = \projcat(A)$, see \cite{B2011}.

\begin{corollary} \label{coro:CM-fin/free}
Let $A=\kk\Q/\I$ be a string algebra, then the following statements hold.
\begin{itemize}
  \item[{\rm(1)}] $A$ is CM-finite.

  \item[{\rm(2)}] $A$ is CM-free if and only if there exists no PPS on the bounded quiver $(\Q, \I)$.
\end{itemize}
\end{corollary}

\begin{proof}
(1) Since string algebra $A = \sum\limits_{\ell=0}^{\infty}\sum\limits_{\mathfrak{p}\in\Q_{\ell}} \kk \mathfrak{p}$ is finite dimensional ($\Q_{\ell}$ is the set of all paths of length $\ell$),
$\Q_{\ge 0} := \bigcup\limits_{\ell=1}^{\infty}\Q_{\ell}$ is a finite set.
On the other hand, by Theorem \ref{thm-Chen}, any NTIG-projective module is isomorphic to $pA$, where $p$ is a perfect path.
Thus, the number of indecomposable G-projective modules (up to isomorphism) less than or equal to the number of elements in $\Q_{\ge 0}$. Thus $A$ is CM-finite.

(2) This is a direct corollary of Theorem \ref{thm-Chen}.
\end{proof}

\subsection{The Cohen-Macaulay Auslander algebras of string algebras}

The {\defines Cohen-Macaulay Auslander algebra} ($=$ CM-Auslander algebra) $B$ of $A$
is the endomorphism algebra of the direct sum of all G-projective $A$-modules which are pairwise not isomorphic, i.e.,
\[ B = \End_A\Big(\bigoplus_{G\in \ind(\Gproj(A))} G\Big) \]
(see \cite[Section 6.2]{B2011}).

In this subsection, we calculate the CM-Auslander algebra of string algebra.
Obviously, the CM-Auslander algebra of a CM-finite algebra is finite dimensional.
Thus, by  Corollary \ref{coro:CM-fin/free}, the CM-Auslander algebra of any string algebra is finite dimensional.

\begin{construction} \rm \label{const:CMA}
Let $A=\kk\Q/I$ with $\Q=(\Q_0, \Q_1)$ be a string algebra and $\scrG^l(A)$ be the set of all NTIG-projective modules which are of the form $pA$, where $p$ is a path of length $l$.
Thus, $\ind(\Gproj(A)) = \bigcup_{l\ge 1} \scrG^l(A)$,
and, for any $pA\in \scrG^l(A)\ne \varnothing$, there exists a relation cycle $\C$ such that $p$ is a path on $\C$ by Lemma \ref{lemm:per.path}.
Define $A^{\CMA} = \kk\Q^{\CMA}/\I^{\CMA}$ whose bounded quiver $(\Q^{\CMA}, \I^{\CMA})$ is constructed by following steps.
\begin{itemize}
  \item[Step 1.]
    Let $\Q^{\CMA}_0 = \Q_0'\cup\Q_0$ with a bijective $\frakv: \ind(\Gproj(A)) \to \Q_0'\cup\Q_0$ sending any indecomposable projective module $P(v)$ to the vertex $v\in \Q_0$, and sending a NTIG-projective module to a vertex in $Q'_0$.

  \item[Step 2.]
    Let \[\Q_1^0 = \Q_1\backslash\{\alpha\in\Q_1 \mid \alpha A\in \scrG^1(A)\}, \]
    and $\Q^1_1$ be the set which is obtained by adding arrows $\alpha^-: s(\alpha) \to \frakv(\alpha A)$ and $\alpha^+: \frakv(\alpha A) \to t(\alpha)$ to $\Q_1^0$ for every $\alpha A \in \scrG^1(A)$.
    In this case, we have $\Q_1^1\supseteq \Q_1^0$ (thus $\Q_1^1 = \Q_1^0 \cup \Q_1^1$).

  \item[Step 3.]
    We define the set $\Q_1^{l}$ by adding following arrows to $\Q_1^{l-1}$. For every $pA \in \scrG^l(A)$,
    \begin{itemize}
      \item  if there is a NTIG-projective module $qA \in \scrG^m(A)$ with $m<l$ such that $s(p)=s(q)$ (resp. $t(p)=t(q)$),
        and there is no NTIG-projective module $rA\in \scrG^n(A)$ with $m < n < l$ such that $s(p)=s(r)$ (resp. $t(p)=t(r)$),
        we adding an arrow $a_{qp}: \frakv(qA) \to \frakv(pA)$ (resp. $a_{pq}: \frakv(pA) \to \frakv(qA)$) to $\Q_1^{l-1}$;
      \item otherwise, we adding a pair of arrows $a_{s(p)p}: s(p) \to \frakv(pA)$
     and  $a_{pt(p)}: \frakv(pA) \to t(p)$ to $\Q_1^{l-1}$.
    \end{itemize}

  \item[Step 4.]
    Let $\Q^{\CMA}_1 = \bigcup_{i\ge 0} \Q_1^i$ and $\Q^{\CMA}$ be the quiver $(\Q^{\CMA}_0, \Q^{\CMA}_1, s, t)$
    with functions $s$ and $t$ sending any arrow in $\Q^{\CMA}_1$ to its starting point and ending point in $\Q^{\CMA}_0$, respectively.

  \item[Step 5.]
   The generators of $\I^{\CMA}$ are given by following three types.
  \begin{itemize}
    \item[($\mathcal{R}$1)]
     For any path $\wp = a_1a_2\cdots a_t$ of length $t$ on $\Q$ which is a generator of $\I$,
     we have $\wp^{*} = a^{*}_1a^{*}_2\cdots a^{*}_t \in\I^{\CMA}$, where
     \[a^{*}_i =
     \begin{cases}
     a_i^{-}a_i^{+}\ \big( {\text{which is a path of} \atop \text{length two on $\Q^{\CMA}$}}\big),
     &\ \text{if}\ a_i \notin \Q^{\CMA}_1 \ \text{and}\ 1 < i < t; \\
     a_1^+, & \ \text{if}\ i=1 \ \text{and}\ a_1\notin \Q^{\CMA}_1; \\
     a_t^-, & \ \text{if}\ i=t \ \text{and}\ a_t\notin \Q^{\CMA}_1; \\
     a_i\ (\in \Q_1), & \ \text{otherwise}. \\
     \end{cases}\]

    \item[($\mathcal{R}$2)]
     Any basic cycle of $\Q^{\CMA}$ with at least one vertex corresponding to NTIG-projective module is commutative.
     That is, for any cycle in $\Q^{\CMA}$ which is one of the following forms
     \begin{center}
       $\xymatrix{
       (\text{i}) & & \frakv(pA) \ar[rd]^{a_{pv_l}} & \\
       & v_0 \ar@{~>}[rr]^{p}_{(\text{a path of length}\ \ge 2)} \ar[ru]^{a_{v_0p}} & & v_l, }$
       $\xymatrix@C=0.5cm{
       (\text{ii}) & & \frakv(qA) \ar[d]_{a_{pv_{m}}} \ar[r]^{a_{qp}}  & \frakv(pA) \ar[d]^{a_{qv_{l}}} \\
       & v_{0} \ar@{~>}[r]^p & v_{m} \ar@{~>}[r]^{p'}_{(pp'=q)} & v_{l}, }$
\\
%
       $\xymatrix@C=0.5cm{
       (\text{iii}) & \frakv(pA) \ar[rr]^{a_{pq}} \ar[d]_{a_{pv}} & & \frakv(qA) \ar[d]^{a_{qr}} \\
        & v \ar[rr]^{a_{vr}} & & \frakv(rA), }$
\ \ \ \
       $\xymatrix@C=0.5cm{
        (\text{iv}) & \frakv(qA) \ar[rr]^{a_{qr}} \ar[d]_{a_{qp}} & & \frakv(rA) \ar[d]^{a_{rs}} \\
        & \frakv(pA)  \ar[rr]^{a_{ps}} & & \frakv(sA), }$
     \end{center}
     the combinations $p - a_{v_{0}p}a_{pv_{l}}$, $a_{pv_{m}}p' - a_{pq}a_{qv_{l}}$,
     $a_{pv}a_{vr} - a_{pq}a_{qr}$ and $a_{qp}a_{ps} - a_{qr}a_{rs}$ in the corresponding diagram are generators in $\I^{\CMA}$.

    \item[($\mathcal{R}$3)]
      For a sequence of arrows
      \[ s(p_1) \mathop{-\!\!\!-\!\!\!\longrightarrow}\limits^{a_{s(p_1)p_1}} \frakv(p_1A)
       \mathop{-\!\!\!-\!\!\!\longrightarrow}\limits^{a_{p_1p_2}} \frakv(p_2A)
        \longrightarrow \cdots
       \mathop{-\!\!\!-\!\!\!\longrightarrow}\limits^{a_{p_{m-1}p_m}} \frakv(p_mA)
       \mathop{-\!\!\!-\!\!\!\longrightarrow}\limits^{a_{p_{m}t(p_m)}} t(p_m) \]
      in $\Q^{\CMA}_1$, if $p_1$, $p_2$, $\cdots$, $p_m$ are paths on $(\Q, \I)$ such that $p_1 \sqcup p_2 \sqcup \cdots \sqcup p_m \in \I$,
      then $a_{s(p_1)p_1}a_{p_1p_2} \cdots a_{p_{m-1}p_m} a_{p_mt(p_m)} \in \I^{\CMA}$.
  \end{itemize}
\end{itemize}
\end{construction}

\begin{theorem} \label{thm:the CMA of string}
Let $A$ be a string algebra, then its CM-Auslander algebra is isomorphic to $A^{\CMA}$.
\end{theorem}

\begin{proof}
Assume $A = \kk\Q^A/\I^A$ and its CM-Auslander algebra is $B = \kk\Q^B/\I^B$, where $\I^B$ is admissible.
By Lemmas \ref{lemm:Gproj(I)}, \ref{lemm:pA-left}, \ref{lemm:pA-right} and \ref{lemm:Gproj-length=1},
and Construction \ref{const:CMA}, we have $\Q^B=\Q^{\CMA}$.

Next, we calculate $\I^B$. Indeed, the generators of $\I^B$ are divided to two classes which are naturally provided by generators of $\I$ and NTIG-projective modules, respectively:
\begin{itemize}
  \item the generators of $\I$ provide a part of relations in $\I^B$ (see (1));

  \item all the others is given by NTIG-projective modules (see (2) and (3)).
\end{itemize}

(1) For any generator $\wp = a_1a_2\cdots a_t$ of length $t$ of $\I^A$,
$a_i$ corresponds to a homomorphism $\hbar_{a_i}: P(t(a_i)) \to P(s(a_i))$ in $\modcat A$.
In $\Gproj(A)$, by Lemma \ref{lemm:Gproj-length=1}, it is not irreducible if and only if $a_iA$ is a NTIG-projective module.
If $\ind(\Gproj(A))$ contains $a_iA$, and $2 < i < t$, then $a_i\notin\Q^{\CMA}_1$ and $a_i^-, a_i^+\in\Q^{\CMA}_1$.
In this case, $\hbar_{a_i}$ has a decomposation
\[\xymatrix{
   P(t(a_i)) \ar[rd]_{\hbar_2} \ar[rr]^{\hbar_{a_i}} & & P(s(a_i)), \\
  & a_iA \ar@{^(->}[ru]_{\hbar_1} &
}\]
where, $\hbar_1$ and $\hbar_2$ are irreducible morphism in $\Gproj(A)$ correspond to $a_{a_is(a_i)} = a_i^{-}$ and $a_{t(a_i)a_i} = a_i^{+}$, respectively.
Define $a^{*}_i =  a_i^{-}a_i^{+}$, if $a_i\notin \Q^{\CMA}_1$ and $1<i<t$;
Otherwise, $a^{*}_i = a_i$. Then $\wp^{*} = a^{*}_1a^{*}_2 \cdots a^{*}_{t-1} a^{*}_t = 0$ on $(\Q^A, \I^A)$,
where
\[a^{*}_1 \ \text{(resp. }\ a^{*}_t \text{)} = \begin{cases}
   a_1^+ \ \text{(resp. }\  a_t^-\text{)},
 & \text{if}\ a_1 \notin \Q^{\CMA}_1 \ \text{(resp. }\ a_t \notin \Q^{\CMA}_1 \text{);} \\
  a_1 \ \text{(resp. }\ a_t \text{)}, & \text{otherwise}.
\end{cases}\]
Thus, it is trivial that $\wp^{*} = 0$ in $(\Q^B, \I^B)$.

(2) Any basic cycle of $\Q^B$ with at least one vertex corresponding to NTIG-projective module is one of four forms (i)-(iv) provided in Construction \ref{const:CMA} Step 5 (2).
Its commutativity is given by the composition of paths.
To be more precise, the commutativity of (i) can be proved by Lemma \ref{lemm:Gproj(I)};
the commutativity of (ii), (iii) and (iv) can be proved by Lemma \ref{lemm:pA-left} or Lemma \ref{lemm:pA-right}.

(3) Let
      \[ s(p_1) \mathop{-\!\!\!-\!\!\!\longrightarrow}\limits^{a_{s(p_1)p_1}} \frakv(p_1A)
       \mathop{-\!\!\!-\!\!\!\longrightarrow}\limits^{a_{p_1p_2}} \frakv(p_2A)
        \longrightarrow \cdots
       \mathop{-\!\!\!-\!\!\!\longrightarrow}\limits^{a_{p_{m-1}p_m}} \frakv(p_mA)
       \mathop{-\!\!\!-\!\!\!\longrightarrow}\limits^{a_{p_{m}t(p_m)}} t(p_m) \]
be a a sequence of arrows in $\Q^{\CMA}_1$ such that $p_1 \sqcup \cdots \sqcup p_m$ is a generator of $\I$,
where $p_i = \alpha_{i,1}\cdots \alpha_{i,\ell_i}$ is a path of length $\ell_i$ ($1\le i\le m$).
We have
\begin{align} \label{formula-1 in thm:the CMA of string}
p_i \sqcap p_{i+1} = \alpha_{i,j_i}\alpha_{i,j_i+1}\cdots \alpha_{i,\ell_i} = \alpha_{i+1,1}\cdots \alpha_{i+1,\ell_i-j_i+1}
\ (1\le j_i\le \ell_i),
\end{align}
or
\begin{align} \label{formula-2 in thm:the CMA of string}
p_i \sqcap p_{i+1} = e_{t(p_i)} = e_{s(p_{i+1})}.
\end{align}
Let $\hbar_{\star}$ be the irreducible morphism in $\Gproj(A)$ corresponding to $a_{\star}$,
where $\star \in \{s(p_1)p_1,$ $ p_1p_2, \cdots, p_{m-1}p_m, p_mt(p_m)\}$.
If (\ref{formula-1 in thm:the CMA of string}) (resp.  (\ref{formula-2 in thm:the CMA of string})) holds,
then $\hbar_{p_{i}p_{i+1}}$ is the morphism induced by the path
\begin{center}
$\tilde{p}_i = \alpha_{i,1}\cdots \alpha_{i,j_i-1}$ (resp.  $\tilde{p}_i = \alpha_{i,1}\cdots \alpha_{i, \ell_i} = p_{i}$),
where $1\le i\le m-1$.
\end{center}
Otherwise, $\hbar_{s(p_{1})p_{1}}$ and $\hbar_{p_{m}t(p_{m})}$ are the morphisms induced by the path $\tilde{p}_0 = e_{s(p_1)}$ and $\tilde{p}_m = p_m$, respectively.
Thus,
\[\hbar = \hbar_{p_mt(p_m)}\hbar_{p_{m-1}p_m}\cdots\hbar_{p_1p_2}\hbar_{s(p_1)p_1}
\in \Hom_A(e_{t(p_m)}A, e_{s(p_1)}A)\]
is the morphism induced by the path
\begin{align}
  & \tilde{p}_0\tilde{p}_1\cdots \tilde{p}_{m-1}\tilde{p}_m \nonumber \\
= & e_{s(p_1)} \cdot \Big( \prod_{i=1}^{m-1} \alpha_{i,1}\cdots \alpha_{i,j_i-1} \Big)
    \cdot (\alpha_{m,1}\cdots \alpha_{m,\ell_m}) \ \  \text{(Note that $\alpha_{i,j_i} = \alpha_{i+1,1}$)} \nonumber \\
=  & p_1\sqcup p_2 \sqcup \cdots \sqcup p_m \in \I.  \nonumber
\end{align}
We obtain $\hbar = 0$, equivalently, $a_{s(p_1)p_1}a_{p_1p_2}\cdots a_{p_{m-1}p_m}a_{p_mt(p_m)}$ is a generator of $\I^{\CMA}$.
\end{proof}

\begin{example} \rm \label{ex:CMA}
Let $A$ be the string algebra given in Example \ref{ex:1}.
There are eight NTIG-projective modules over $A$:
\begin{align}
 & abc A \cong S(4) = (4), && dA \cong  \left(\begin{smallmatrix} 5 \\ 6 \\ 7 \end{smallmatrix}\right),
&& efg A \cong S(8) = (8), && hA \cong \left(\begin{smallmatrix} 1 \\ 2 \\ 3 \end{smallmatrix}\right), \nonumber \\
 & ab A \cong \left({^3_4}\right), && cd A \cong \left({^5_6}\right),
&& ef A \cong \left({^7_8}\right), && gh A \cong \left({^1_2}\right).  \nonumber
\end{align}
The CM-Auslander algebra of $A$ is $A^{\CMA} = \kk\Q^{\CMA}/\I^{\CMA}$, where $\Q^{\CMA}$ is shown in \Pic \ref{fig:CMA of A in ex} and
\begin{align}
\I^{\CMA} = & \langle abcd^{-}, cd^{-}d^{+}ef, efgh^{-}, gh^{-}h^{+}ab, h^{+}abc, d^{+}efg, ax, ey, \nonumber \\
& \beta_1\alpha_2, \beta_2d^{-}\alpha_3, \beta_3\alpha_4, \beta_4h^{+}\alpha_1,
 h^{+}\alpha_1\gamma_1, \delta_1d^{-}, d^{+}\alpha_3\gamma_2, \delta_2h^{-}, \nonumber \\
& \alpha_1\beta_1-ab, \alpha_2\beta_2-cd^{-}, \alpha_3\beta_3-ef, \alpha_4\beta_4-gh^{-}, \nonumber \\
& \gamma_1\delta_1-\beta_1 c, \gamma_2\delta_2 - \beta_3 g \rangle. \nonumber
\end{align}
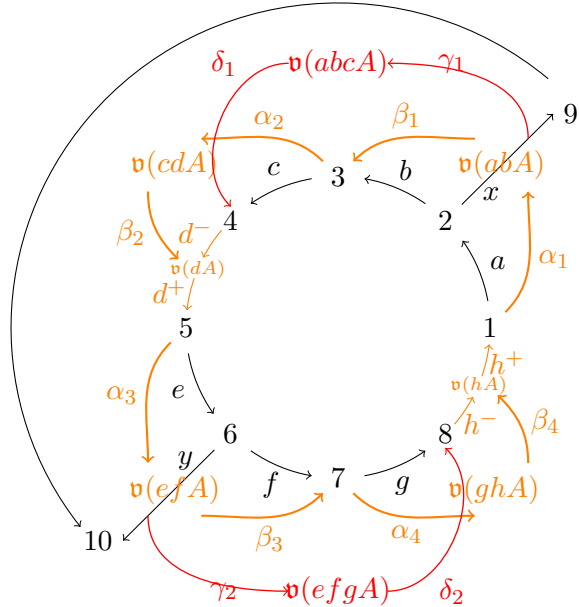
\begin{figure}[htbp]
\centering
\begin{tikzpicture}
\draw[->][rotate=  0+10] (2,0) arc (0:25:2);
\draw[->][rotate= 45+10] (2,0) arc (0:25:2);
\draw[->][rotate= 90+10] (2,0) arc (0:25:2);
\draw[orange][->][rotate=135+5] (2,0) arc (0:12.5:2);
\draw[orange][->][rotate=135+26] (2,0) arc (0:12.5:2);
\draw[->][rotate=180+10] (2,0) arc (0:25:2);
\draw[->][rotate=225+10] (2,0) arc (0:25:2);
\draw[->][rotate=270+10] (2,0) arc (0:25:2);
\draw[orange][->][rotate=315+5] (2,0) arc (0:12.5:2);
\draw[orange][->][rotate=315+26] (2,0) arc (0:12.5:2);
\draw[rotate=  0] (2,0) node{$1$}; \draw[rotate=  0] (2.1,0.85) node{$a$};
\draw[rotate= 45] (2,0) node{$2$}; \draw[rotate= 45] (2.1,0.85) node{$b$};
\draw[rotate= 90] (2,0) node{$3$}; \draw[rotate= 90] (2.1,0.85) node{$c$};
\draw[rotate=135] (2,0) node{$4$}; 
\draw[orange][rotate=135+22.5] (2,0) node{\tiny $\frakv(dA)$};
\draw[orange][rotate=135-11] (2.1,0.85) node{$d^{-}$};
\draw[orange][rotate=135+11] (2.1,0.85) node{$d^{+}$};
\draw[rotate=180] (2,0) node{$5$}; \draw[rotate=180] (2.1,0.85) node{$e$};
\draw[rotate=225] (2,0) node{$6$}; \draw[rotate=225] (2.1,0.85) node{$f$};
\draw[rotate=270] (2,0) node{$7$}; \draw[rotate=270] (2.1,0.85) node{$g$};
\draw[rotate=315] (2,0) node{$8$}; 
\draw[orange][rotate=315+22.5] (2,0) node{\tiny $\frakv(hA)$};
\draw[orange][rotate=315-11] (2.1,0.85) node{$h^{-}$};
\draw[orange][rotate=315+11] (2.1,0.85) node{$h^{+}$};
\draw[rotate= 45][->] (2.3,0) -- (4,0) node[right]{$9$};
\draw[rotate= 45] (2.5,0) node[right]{$x$};
\draw[rotate= 45+180][->] (2.3,0) -- (4,0) node[left]{$10$};
\draw[rotate= 45+180] (2.5,0) node[left]{$y$};
\draw[rotate= 45+5][->] (4.3,0) arc (0:168:4.3);
\draw[orange][->][line width=0.8pt] (2.2,0.2) to[out=45, in=-90](2.5,1.8);
\draw[orange] (2.5,0.9) node[right]{$\alpha_1$};
\draw[orange][->][line width=0.8pt] (1.8,2.5) to[out=180,in=45] (0.2,2.2);
\draw[orange] (0.9,2.5) node[above]{$\beta_1$};
\draw[orange] (2.15,2.15) node{$\frakv(abA)$};
\draw[orange][->][line width=0.8pt] (-0.2,2.2) to[out=135,in=  0] (-1.8,2.5);
\draw[orange](-0.9,2.5) node[above]{$\alpha_2$};
\draw[orange][->][line width=0.8pt] (-2.5,1.8) to[out=-90,in=135] (-2.1,0.9);
\draw[orange] (-2.4,1.25) node[left]{$\beta_2$};
\draw[orange] (-2.15,2.15) node{$\frakv(cdA)$};
\draw[rotate=180][orange][->][line width=0.8pt] (2.2,0.2) to[out=45, in=-90](2.5,1.8);
\draw[rotate=180][orange] (2.5,0.9) node[left]{$\alpha_3$};
\draw[rotate=180][orange][->][line width=0.8pt] (1.8,2.5) to[out=180,in=45] (0.2,2.2);
\draw[rotate=180][orange] (0.9,2.5) node[below]{$\beta_3$};
\draw[rotate=180][orange] (2.15,2.15) node{$\frakv(efA)$};
\draw[rotate=180][orange][->][line width=0.8pt] (-0.2,2.2) to[out=135,in=  0] (-1.8,2.5);
\draw[rotate=180][orange](-0.9,2.5) node[below]{$\alpha_4$};
\draw[rotate=180][orange][->][line width=0.8pt] (-2.5,1.8) to[out=-90,in=135] (-2.1,0.9);
\draw[rotate=180][orange] (-2.4,1.25) node[right]{$\beta_4$};
\draw[rotate=180][orange] (-2.05,2.15) node{$\frakv(ghA)$};
\draw[red] (0,3.5) node{$\frakv(abcA)$};
\draw[red][->][line width=0.5pt] (2.5,2.5) to[out=90,in=0] (0.65,3.5);
\draw[red][->][line width=0.5pt] (-0.65,3.5) to[out=180,in=135] (-1.41,1.6);
\draw[red] ( 1.5,3.5) node{$\gamma_1$};
\draw[red] (-1.5,3.5) node{$\delta_1$};
\draw[rotate=180][red] (0,3.5) node{$\frakv(efgA)$};
\draw[rotate=180][red][->][line width=0.5pt] (2.5,2.5) to[out=90,in=0] (0.65,3.5);
\draw[rotate=180][red][->][line width=0.5pt] (-0.65,3.5) to[out=180,in=135] (-1.41,1.6);
\draw[rotate=180][red] ( 1.5,3.5) node{$\gamma_2$};
\draw[rotate=180][red] (-1.5,3.5) node{$\delta_2$};
\end{tikzpicture}
\caption{The CM-Auslander algebra of the algebra in Example \ref{ex:1}}
\label{fig:CMA of A in ex}
\end{figure}
\end{example}

\begin{example} \rm \label{ex:self-inj}
Let $A = \kk\Q/\I$ be the string algebra whose bounded quiver is given by $\Q=$
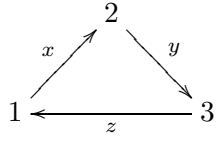
\[\xymatrix{
 & 2 \ar[rd]^{y} & \\
1 \ar[ru]^{x} & & 3 \ar[ll]^{z} }\]
and $\I = \{xyzx, yzxy, zxyz\}$.
We can check that $A$ is self-injective. Then all indecomposable modules are G-projective, that is,
\begin{align}
\ind(\modcat A) & = \ind(\Gproj(A)) \nonumber \\
& = \{P(1), P(2), P(3), xA, yA, zA, xyA, yzA, zxA, xyzA, yzxA, zxyA\}. \nonumber
\end{align}
The CM-Auslander algebra $A^{\CMA}$ is isomorphic to $\kk\Q^{\CMA}/\I^{\CMA}$, where $\Q^{\CMA}$ is shown in \Pic \ref{fig:CMA of A in ex:self-inj}, and $\I^{\CMA}$ is generated by
\begin{align}
 & a_{x2}a_{2y}a_{y3}a_{3z}a_{z1}a_{1x},
&& a_{y3}a_{3z}a_{z1}a_{1x}a_{x2}a_{2y},
&& a_{z1}a_{1x}a_{x2}a_{2y}a_{y3}a_{3z},  \nonumber \\
 & a_{x2}a_{2y} - a_{x,xy}a_{xy,y},
&& a_{y3}a_{3z} - a_{y,yz}a_{yz,z},
&& a_{z1}a_{1x} - a_{z,zx}a_{zx,x}, \nonumber \\
 & a_{zx,x}a_{x,xy} - a_{zx,zxy}a_{zxy,xy},
&& a_{xy,y}a_{y,yz} - a_{xy,xyz}a_{xyz,yz},
&& a_{yz,z}a_{z,zx} - a_{yz,yzx}a_{yzx,zx}, \nonumber \\
 & a_{zxy,xy}a_{xy,xyz},
&& a_{xyz,yz}a_{yz,yzx},
&& \text{\ and\ } a_{yzx,zx}a_{zx,zxy}. \nonumber
\end{align}
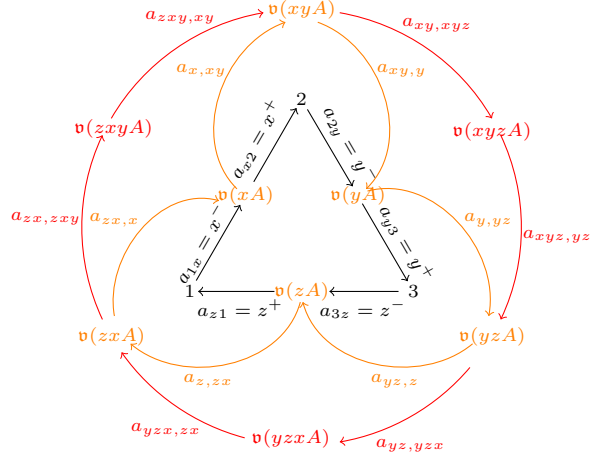
\begin{figure}[htbp]
\centering
\begin{tikzpicture} [scale=0.85] \tiny
\draw [->] (-1.63,-0.83) -- (-0.95, 0.35);
\draw [->] (-0.74, 0.71) -- (-0.08, 1.86);
\draw [->] ( 0.09, 1.84) -- ( 0.79, 0.64);
\draw [->] ( 0.94, 0.37) -- ( 1.64,-0.84);
\draw [->] ( 1.5 ,-1   ) -- ( 0.43,-1   );
\draw [->] (-0.42,-1   ) -- (-1.6 ,-1   );
\draw (-1.55,-0.25) node{{\rotatebox{60}{$a_{1x} =x^{-}$}}};
\draw (-0.71, 1.38) node{{\rotatebox{60}{$a_{x2} =x^{+}$}}};
\draw ( 0.75, 1.28) node{{\rotatebox{-60}{$a_{2y} = y^{-}$}}};
\draw ( 1.61,-0.25) node{{\rotatebox{-60}{$a_{y3} = y^{+}$}}};
\draw (-0.95, -1  ) node[below]{$a_{z1} = z^{+}$};
\draw ( 0.95, -1  ) node[below]{$a_{3z} = z^{-}$};
\draw ( 0   , 2   ) node{$2$};
\draw ( 1.73,-1   ) node{$3$};
\draw (-1.73,-1   ) node{$1$};
\draw[orange] (-0.87, 0.5 ) node{$\frakv(xA)$};
\draw[orange] ( 0.87, 0.5 ) node{$\frakv(yA)$};
\draw[orange] ( 0   ,-1   ) node{$\frakv(zA)$};
\draw[orange] ( 0   , 3.40) node{$\frakv(xyA)$};
\draw[orange] (-1.55, 2.45) node{$a_{x,xy}$};
\draw[orange] ( 1.55, 2.45) node{$a_{xy,y}$};
\draw[orange][->] (-1.00, 0.60) to[out=135, in=200] (-0.25, 3.20);
\draw[orange][->] ( 0.25, 3.20) to[out=-20, in= 45] ( 1.00, 0.60);
\draw[orange][rotate=120] ( 0   , 3.40) node{$\frakv(zxA)$};
\draw[orange][rotate=120] (-1.35, 2.45) node{$a_{z,zx}$};
\draw[orange][rotate=120] ( 1.55, 2.45) node{$a_{zx,x}$};
\draw[orange][rotate=120] [->] (-1.00, 0.60) to[out=135, in=200] (-0.25, 3.20);
\draw[orange][rotate=120] [->] ( 0.25, 3.20) to[out=-20, in= 45] ( 1.00, 0.90);
\draw[orange][rotate=240] ( 0   , 3.40) node{$\frakv(yzA)$};
\draw[orange][rotate=240] (-1.55, 2.45) node{$a_{y,yz}$};
\draw[orange][rotate=240] ( 1.35, 2.45) node{$a_{yz,z}$};
\draw[orange][rotate=240] [->] (-1.00, 0.60) to[out=135, in=200] (-0.25, 3.20);
\draw[orange][rotate=240] [->] ( 0.25, 3.20) to[out=-20, in= 45] ( 1.00, 0.60);
\draw[red] ( 2.94, 1.50) node{$\frakv(xyzA)$};
\draw[red][->][rotate=-10] ( 0.00, 3.40) arc(90:40:3.40);
\draw[red][->][rotate=-65] ( 0.00, 3.40) arc(90:40:3.40);
\draw[red][rotate=-120] ( 2.94, 1.50) node{$\frakv(yzxA)$};
\draw[red][->][rotate=-10][rotate=-120] ( 0.00, 3.40) arc(90:50:3.40);
\draw[red][->][rotate=-75][rotate=-120] ( 0.00, 3.40) arc(90:50:3.40);
\draw[red][rotate=   5][rotate=-240] ( 2.94, 1.50) node{$\frakv(zxyA)$};
\draw[red][->][rotate= -5][rotate=-240] ( 0.00, 3.40) arc(90:40:3.40);
\draw[red][->][rotate=-60][rotate=-240] ( 0.00, 3.40) arc(90:40:3.40);
\draw[red] ( 2.10, 3.16) node{$a_{xy, xyz}$};
\draw[red][rotate= -60] ( 2.10, 3.36) node{$a_{xyz,yz}$};
\draw[red][rotate=-120] ( 2.10, 3.16) node{$a_{yz,yzx}$};
\draw[red][rotate=-180] ( 2.10, 3.16) node{$a_{yzx,zx}$};
\draw[red][rotate=-240] ( 2.10, 3.36) node{$a_{zx,zxy}$};
\draw[red][rotate=-300] ( 1.90, 3.26) node{$a_{zxy,xy}$};
\end{tikzpicture}
\caption{The CM-Auslander algebra of the string algebra in Example \ref{ex:self-inj}}
\label{fig:CMA of A in ex:self-inj}
\end{figure}
Note that other paths satisfying ($\mathcal{R}$3) (see Construction \ref{const:CMA} Step 5)  are elements in $\I^{\CMA}$.
For example,
\begin{center}
$a_{x,xy}a_{xy,y}a_{y,yz}a_{yz,z}a_{z,zx}
= (a_{x,xy}a_{xy,xyz})\cdot(a_{xyz,yz}a_{yz,yzx})\cdot a_{yzx,zx} \in \I^{CMA}$
\end{center}
because $a_{xyz,yz}a_{yz,yzx} \in \I^{\CMA}$.
\end{example}

\section{Applications} \label{sec:appl}

\subsection{The string algebras satisfying G-condition} \label{subsec:G-cond}
In this part, we mainly consider the CM-Auslander algebras of a class of special string algebras.

A relation cycle $\C$ of string algebra $A=\kk\Q/\I$ is called a {\defines gentle relation cycle} (or a {\defines forbidden cycle}) if all relations on $\C$ are paths of length two.
Furthermore, we say a string algebra $A$ satisfies {\defines G-condition} if all relation cycles providing NTIG-projective modules are gentle relation cycles.

\begin{lemma}
Let $A$ be a string algebra satisfying G-condition such that its bounded quiver has at least one gentle relation cycle.
Then any NTIG-projective module in $\Gproj(A)$ is isomorphic to $aA$ where $a$ is an arrow on a gentle relation cycle.
\end{lemma}

\begin{proof}
Assume $A=\kk\Q/\I$ with a gentle relation cycle
\[\C = a_0a_1\cdots a_{n-1}\ (s(a_i) = i, 0\le i\le n-1), \]
where $a_{\overline{j}}a_{\overline{j+1}}\in \I$ ($\overline{x}$ equals to $x$ modulo $n$).
By Construction \ref{const:perpaths}, the above relation cycle provide $n$ PPSs which can be constructed by the PRSs
\begin{center}
$\{(a_{\overline{j}}a_{\overline{j+1}},
a_{\overline{j+2}}a_{\overline{j+3}},
\cdots,
a_{\overline{j+n-2}}a_{\overline{j+n-1}})\}_{0\le j\le n-1}$.
\end{center}
Then $\C$ provides $n$ NTIG-projective modules $\{a_jA\}_{0\le j\le n-1}$.
\end{proof}

Then the following result can be deduced directly by Theorem \ref{thm:the CMA of string}.

\begin{corollary} \label{coro:the CMA of string G}
Let $A=\kk\Q/\I$ be a string algebra.
\begin{itemize}
\item[{\rm (1)}]
If $A$ satisfies G-condition, then $A^{\CMA}\cong\Q^{\CMA}/\I^{\CMA}$ can be calculated by the following steps.
\begin{itemize}
  \item[{\rm Step 1.}] $\Q^{\CMA}$ is obtained by splitting any arrow $\alpha: x\to y$ on gentle relation cycle of $(\Q, \I)$
      into two arrows $\alpha^{-}: x\to \frakv(\alpha A)$ and $\alpha^{+}: \frakv(\alpha A) \to y$.
  \item[{\rm Step 2.}] $\I^{\CMA}$ is generated by the following two classes of paths:
  \begin{itemize}
    \item[{\rm(i)}] all paths of length two which are of the form
     $\alpha^{+}\beta^{-}$ {\rm(}$\alpha$ and $\beta$ are arrows on the same gentle relation cycle of $(\Q, \I)${\rm)};
    \item[{\rm(ii)}] all paths which are of the form $\wp^{*} = \alpha^{*}_1 \cdots \alpha^{*}_l$, where $\wp = \alpha_1 \cdots \alpha_l\in \I$.
  \end{itemize}
\end{itemize}

\item[{\rm(2)}]
 $A^{\CMA}$ is a string algebra if and only if $A$ satifies G-condition. Moreover, in this case,  $A^{\CMA}$ is a string algebra satisfying G-condition whose all relation cycles do not provide NTIG-projective module.
\end{itemize}
\end{corollary}

Thanks to \cite{BR1987}, the string algebras are known as an important  class of tame algebras, and whether it is representation-finite or not is determined by that if it has a band or not.  The following result shows that the representation type of any string algebra satisfying G-condition can be determined by its CM-Auslander algebra.

\begin{theorem} \label{thm:G-condition}
A string algebra satisfying G-condition is representation-finite if and only if its CM-Auslander algebra is representation-finite.
\end{theorem}

\begin{proof}
If a string algebra $A=\kk\Q/\I$ satisfying G-condition is representation-infinite,
then there exists a band $b$ on the bounded quiver $(\Q, \I)$ (see \cite[Lemma 2.2]{M2019}, or c.f. \cite{CB1995, VFCB1998}).
Thus, we have two cases as follows.
\begin{itemize}
  \item If there exists an arrow $\alpha$ on $b$ which is also an arrow on a relation cycle $\C$ of $(\Q, \I)$,
    then $\alpha$ corresponds to the path $\alpha^{*} = \alpha^{-}\alpha^{+}$ of length two on $(\Q^{\CMA}, \I^{\CMA})$.
    In this case, $b^{*}$ is also a band on $(\Q^{\CMA}, \I^{\CMA})$.

  \item Otherwise, $b$ also occur as a band in $(\Q^{\CMA}, \I^{\CMA})$.
\end{itemize}
Therefore, $A^{\CMA}$ is representation-infinite.

If $A=\kk\Q/\I$ is a representation-finite algebra, then for any circuit $\C$ of $(\Q, \I)$,
there exists a path $\wp$ on $\C$ such that $\wp$ is a generator of $(\Q, \I)$.
It is holds that $\wp^{*} \in \I^{\CMA}$ is a path on the circuit $\C^{*}$ of $(\Q^{\CMA}, \I^{\CMA})$.
By Corollary \ref{coro:the CMA of string G}, any circuit of $(\Q^{\CMA}, \I^{\CMA})$ are corresponded by some circuit of $(\Q, \I)$, and this correspondence is bijective.
Thus $(\Q^{\CMA}, \I^{\CMA})$ has no band, i.e., $A^{\CMA}$ is representation-finite.
\end{proof}

\begin{remark} \rm \label{rmk:repr-type}
For general representation-finite string algebra $A$, its CM-Auslander algebra $A^{\CMA}$ may be representation-infinite.
For instance, the string algebras provided in Examples \ref{ex:1} and \ref{ex:self-inj} are representation-finite,
but their CM-Auslander algebras are representation-infinite.
\end{remark}

Suppose that $A$ is a string algebra. On one hand, if $A$ satisfies G-condition, then by the proof of Theorem \ref{thm:G-condition}, it is easy to see that the CM-Auslander algebra $A^{\CMA} = \kk\Q^{\CMA}/\I^{\CMA}$ of $A$ is a string algebra.
Conversely, by the construction of CM-Auslander algebra, if the CM-Auslander algebra $A^{\CMA}$ of $A$ is not a string algebra, then $\Gproj(A)$ contains at least one non-trivial indecomposable G-projective module which is of the form $pA$, where $p$ is a path of length $l(p)\ge 2$.
In this case, $\I^{\CMA}$ contains at least one commutative relation (see Construction \ref{const:CMA}, Step 5 ($\mathcal{R}$2)). Thus, we have the following result.

\begin{corollary}
Let $A$ be a string algebra. Then $A^{\CMA}$ is a string algebra if and only if $A$ is a string algebra satisfying G-condition. 
\end{corollary}

\subsection{The Cohen-Macaulay Auslander algebras of gentle algebras}
Recall that a finite dimensional $\kk$-algebra is called a gentle algebra if
its bounded quiver is gentle (see Section \ref{sec:string and gentle}).
Obviously, any gentle algebra is also a string algebra. Then the CM-Auslander algebra $A^{\CMA}$ of gentle algebra $A$
can be calculated by Corollary \ref{coro:the CMA of string G},
where the length of every path $\wp^{*}$ as described in the Step 2 (ii) of
Corollary \ref{coro:the CMA of string G} is equal to two.
Furthermore, by the definition of gentle algebra, we obtain the following corollary
which is originally proved by Chen and Lu in \cite[Theorem 3.5]{CL2019}.

\begin{corollary} \label{coro:the CMA of gentle}
Let $A=\kk\Q/\I$ be a gentle algebra, then $A^{\CMA}\cong\Q^{\CMA}/\I^{\CMA}$,
where $(\Q^{\CMA},\I^{\CMA})$ can be calculated by the following steps.
\begin{itemize}
  \item[{\rm Step 1.}] $\Q^{\CMA}$ is obtained by splitting every arrow $\alpha: x\to y$ on gentle relation cycle of $(\Q, \I)$
      to two arrows $\alpha^{-}: x\to \frakv(\alpha A)$ and $\alpha^{+}: \frakv(\alpha A) \to y$.
  \item[{\rm Step 2.}] All generators of $\I$ are paths of length two which can be divided to two cases:
  \begin{itemize}
    \item[{\rm Case 1.}] the path $a_1a_2$ of length two on gentle relation cycle;

    \item[{\rm Case 2.}] the path $a_1a_2$ of length two whose arrows $a_1$ and $a_2$ are not  on gentle relation cycle,
  \end{itemize}
  and
  \[ \I^{\CMA} = \langle \wp = a_1a_2 \mid \text{$\wp$ is a path satisfying {\rm Case 1} or {\rm Case 2}} \rangle. \]
\end{itemize}
Furthermore, $A^{\CMA}$ is a gentle algebra.
\end{corollary}

Since gentle algebras are string algebras satisfying G-condition, the following corollary can be directly deduced by Theorem \ref{thm:G-condition}, see also \cite{CL2019}.

\begin{corollary}
A gentle algebra is representation-finite if and only if its CM-Auslander algebra is representation-finite.
\end{corollary}


 Now, we observe the derived representation type of a gentle algebra and its CM-Auslander algebra. Let $X^{\bullet} = (X^t, d^t)$ of the form of
\[ \xymatrix{ \cdots \ar[r] & X^{-1} \ar[r]^{d^{-1}} & X^0 \ar[r]^{d^0} & X^1 \ar[r] & \cdots} \]
be an indecomposable complex in the derived category $D^b(A)$ of a finite dimensional $\kk$-algebra $A$.
The {\defines cohomological dimension vector} of $X^{\bullet}$ is the vector
\begin{center}
  $\pmb{d}(X^{\bullet}) := (\dim_{\kk}\HH^n(X^{\bullet}))_{n\in \ZZ} \in \NN^{(\ZZ)}$,
\end{center} 
where $\HH^n(X^{\bullet}) = \mathrm{Ker}(d^n)/ \mathrm{Im}(d^{n-1})$ is the $n$-cohomology of $X^{\bullet}$.
Furthermore, we define that
\begin{itemize}
  \item the {\defines cohomological length} $\hl(X^{\bullet})$ of $X^{\bullet}$ is
    $\max\{\dim_{\kk}\HH^i(X^{\bullet}) \mid i\in \ZZ\}$,
  \item the {\defines cohomological width} $\hw(X^{\bullet})$ of $X^{\bullet}$ is 
    $\max\{j-i+1 \mid \HH^i(X^{\bullet}) \ne 0 \ne \HH^j(X^{\bullet})\}$, 
  \item the {\defines cohomological range} $\hr(X^{\bullet})$ of $X^{\bullet}$ is the product
    $\hl(X^{\bullet}) \cdot \hw(X^{\bullet})$, 
\end{itemize}
and we call that a finite dimensional $\kk$-algebra $A$ is 
\begin{itemize}
  \item {\defines derived discrete}, if for any $\pmb{d}\in \NN^{(\ZZ)}$, there are only finite indecomposable objects in $D^b(A)$ (up to isomorphism) whose cohomological dimension vector equals to $\pmb{d}$.
  \item {\defines strongly derived unbounded}, if there is a sequence $r_1, r_2, \cdots \in \NN$ such that, for any $i\in \NN$, the number of indecomposable objects in $D^b(A)$ with cohomological range being $r_i$ is infinite (up to isomorphism and shift).
\end{itemize}

In [ZH16], it was shown that any algebra is either a derived discrete algebra or a strongly derived unbounded algebra. To be more precise,  any gentle algebra is derived discrete if and only if its bounded quiver has no homotopy band (see \cite[Section 2.2]{ALP2016} for the definition). Thus, by the proof of Theorem \ref{thm:G-condition}, the following result holds.

\begin{corollary}
Let $A$ be a gentle algebra. Then
\begin{itemize}
  \item[\rm(1)] $A$ is derived discrete if and only if $A^{\CMA}$ is derived discrete. 
  \item[\rm(2)] $A$ is strongly derived unbounded if and only if $A^{\CMA}$ is strongly derived unbounded.
\end{itemize}
\end{corollary}

\subsection{On homological dimensions of gentle algebras}

Recall form \cite[Section 2.2]{AG2008} that,
a {\defines non-trivial forbidden path} $F = \alpha_1\cdots\alpha_n$ of gentle algebra $A = \kk\Q/\I$
is a path on the quiver $\Q$ such that $\alpha_{i}\alpha_{i+1} \in \I$.
A {\defines trivial forbidden path} is a trivial path corresponded by $v\in\Q_0$ such that there exists at most one arrow ending at $v$ and there exists at most one arrow starting at $v$, and if $\alpha, \beta \in \Q_1$ are such that $t(\alpha) = v = s(\beta)$ then $\alpha\beta \in \I$.
We say a forbidden path $F$ is a {\defines forbidden thread} if one of the following holds.
\begin{itemize}
  \item For any arrows $\beta$ and $\beta'$ such that $t(\beta) = s(F)$ and $t(F) = s(\beta')$
    we have $\beta\alpha_1 \notin \I$ and $\alpha_n\beta' \notin \I$, respectively.
  \item  $F$ is a trivial forbidden path ($F$ is called a trivial forbidden thread in this case).
\end{itemize}

\begin{lemma} \label{lemm:forbidden}
Let $A$ be a gentle algebra with at least one gentle relation cycle $\C$ and $F=\alpha_1\cdots\alpha_n$ be a forbidden path of length $n$.
If there exists an integer $i$ with $1\le i\le n$ such that $\alpha_i$ is an arrow on $\C$,
then $F$ is a forbidden path on $\C$.
\end{lemma}

\begin{proof}
Assume $\C$ is of the form
\[\C = c_0c_1\cdots c_{\ell-1}\ \text{whose length is\ } \ell.\]
Then there exists an integer $j$ with $1\le j\le \ell$ such that $\alpha_i = c_j$.
If $\alpha_{i+1} \ne c_{\overline{j+1}}$ ($i\le n-1$), where $\overline{x}$ is $x$ modulo $\ell$,
then $\alpha_{i+1}$ is an arrow starting at $t(\alpha_i)=t(c_j)$.
By the definition of gentle algebra, we have $\alpha_i\alpha_{i+1} = c_{j}\alpha_{i+1} \notin \I$
because $\C$ is a gentle relation cycle, we obtain a contradiction.
Thus $\alpha_{i+1} = c_{\overline{j+1}}$ is an arrow on $\C$.
Similarly, we have $\alpha_{i-1} = c_{\overline{j-1}}$ if $i\ge 2$.
We can show that this lemma by induction.
\end{proof}

The following lemma shows that the global dimension $\gldim A$ and self-injective dimension $\id A$ of gentle algebra $A$ can be described by forbidden paths and forbidden threads, respectively.

\begin{lemma} \label{lemm:dim}
Let $A=\kk\Q/\I$ be a gentle algebra, $\mathfrak{f}(A)$ be the set of all forbidden paths of $A$,
and $\mathfrak{F}(A)$ be the set of all forbidden threads of $A$.
Then we have
\begin{itemize}
  \item[{\rm(1)}]
    \[\gldim A = \sup_{F\in\mathfrak{f}(A)} \ell(F) =
    \begin{cases}
    \infty, & \text{$A$ has at least one gentle relation cycle}; \\
    \sup\limits_{F\in\mathfrak{F}(A)} \ell(F), & otherwise;
    \end{cases} \]

  \item[{\rm(2)}] {\rm \cite[Theorem 3.4]{GR2005}}
    \[\id A = \sup_{F\in\mathfrak{F}(A)} \ell(F).\]
\end{itemize}
\end{lemma}

By Lemma \ref{lemm:forbidden}, all forbidden paths of gentle algebra $A$ can be divided to two classes:
\begin{itemize}
  \item[(I)] the forbidden paths which are on gentle relation cycles;
  \item[(II)] the forbidden paths whose all arrows are not arrows of any gentle relation cycles.
\end{itemize}
We denote by $\mathfrak{f}_1(A)$ (resp.  $\mathfrak{f}_2(A)$) the set of all forbidden paths belong in (I) (resp.  (II)),
then $\mathfrak{f}_1(A)\cup \mathfrak{f}_2(A) = \mathfrak{f}(A)$. Furthermore, it is easy to see that
\begin{align}\label{formula:idA}
\id A  = \sup_{F\in\mathfrak{f}_2(A)} \ell(F)
\end{align}
by Lemma \ref{lemm:dim}.
Therefore, we immediately obtain the description of the self-injective dimension of any gentle algebra $A$
by the forbidden threads of its CM-Auslander algebra $A^{\CMA}$.

Furthermore, we have the following corollary.

\begin{corollary} \label{coro:fdim=gldim}
Let $A$ be a gentle algebra.
\begin{itemize}
  \item[{\rm(1)}] If one of following condition holds
    \begin{itemize}
      \item $\gldim A < \infty$,
      \item $\gldim A = \infty$ and $\gldim A^{\CMA} \ge 3$,
    \end{itemize}
   then \[\id A^{\CMA} \ (= \gldim A^{\CMA})\ = \id A.\]
  \item[{\rm(2)}] Otherwise, \[\id A^{\CMA} \ (=\gldim A^{\CMA})\ = 2 \ge \id A.\]
\end{itemize}
\end{corollary}

\begin{proof}
Since $A$ is a gentle algebra, so is $A^{\CMA}$ by Corollary \ref{coro:the CMA of gentle}.
All forbidden paths of $A^{\CMA}$ can be decided by
\begin{align}\label{formula in thm:dim}
\mathfrak{f}(A^{\CMA}) = (\mathfrak{f}_1(A))^{*} \cup \mathfrak{f}_2(A) = \mathfrak{f}_2(A^{\CMA})
\end{align}
where
\[(\mathfrak{f}_1(A))^{*} = \bigcup\limits_{F = \alpha_1\cdots\alpha_n \in \mathfrak{f}_1(A)}\{\alpha_1^{-}\alpha_1^{+}, \cdots, \alpha_n^{-}\alpha_n^{+}\},
\text{\ and\ } \mathfrak{f}_1(A^{\CMA}) = \varnothing.\]

If $\gldim A < \infty$, then $A\cong A^{\CMA}$. In this case, $\id A = \gldim A = \gldim A^{\CMA}$.

If $\gldim A = \infty$, then $A$ has gentle relation cycles, say $\C_1, \cdots, \C_r$.
Assume $\C_i = c_{i,1}\cdots c_{i,l_i}$ and $s(c_{i,j}) = v_{i,j}$ ($1\le i\le r$, $1\le j\le l_i$).
Then the quiver $\Q^{\CMA}$ of $A^{\CMA}$ can be obtained by splitting all arrows $c_{i,j}: v_{i,j} \to v_{i,j+1}$
to two arrows $c_{i,j}^{-}: v_{i,j} \to \frakv(c_{i,j}A)$ and $c_{i,j}^{-}: \frakv(c_{i,j}A) \to v_{i,j+1}$,
and the admissible ideal $\I^{\CMA}$ can be obtained by replacing $c_{i,j}c_{i,j+1}$ to $c_{i,j}^{+}c_{i,j+1}^{-}$.
Thus, any gentle relation cycle $\C_i$ of $A$ changes to an oriented cycle $\C_i^{*} = c_{i,1}^{-}c_{i,1}^{+}\cdots c_{i,l_i}^{-}c_{i,l_i}^{+}$ of $A^{\CMA}$ whose length is $2l_i$,
and, by (\ref{formula in thm:dim}), the length of arbitrary forbidden paths on $\C_i^{*}$ is a path in $(\mathfrak{f}_1(A))^*$ with length two.
Thus, by Lemma \ref{lemm:dim} (1), we have
\[\gldim A^{\CMA} = \sup_{F\in\mathfrak{f}(A^{\CMA})} \ell(F)
= \sup_{F\in\mathfrak{f}_2(A^{\CMA})} \ell(F) \ge 2.\]
If $\gldim A^{\CMA} \ge 3$, then
\[ \id A \mathop{=\!=\!=}\limits^{(\ref{formula:idA})} \sup_{F\in\mathfrak{f}_2(A)} \ell(F)
= \sup_{F\in\mathfrak{f}_2(A)\cup(\mathfrak{f}_1(A))^{*}} \ell(F)
\mathop{=\!=\!=}\limits^{(\ref{formula in thm:dim})} \gldim A^{\CMA}. \]
Thus (1) holds.
Otherwise,
\[2 = \gldim A^{\CMA} = \sup_{F\in\mathfrak{f}_2(A^{\CMA})} \ell(F)
= \sup_{F\in(\mathfrak{f}_1(A))^{*}} \ell(F) \ge \sup_{F\in\mathfrak{f}_2(A)} \ell(F) = \id A. \]
Notice that $\gldim A^{\CMA} < \infty$, thus $\id A^{\CMA} = \gldim A^{\CMA}$.
We obtain (2).
\end{proof}

\begin{remark} \rm
The more general case of Corollary \ref{coro:fdim=gldim} (1) is
\begin{center}
$\gldim \Lambda^{\CMA} = \max\{\id ({_\Lambda\Lambda}), \id (\Lambda_\Lambda)\}$,
\end{center}
where $\Lambda$ is an any finite CM-type Artin algebra whose the Gorenstein global dimension of $\Lambda$ $\mathrm{G}\text{-}\dim (\modcat \Lambda)$ is greater than or equal to $3$;
and Corollary \ref{coro:fdim=gldim} (2) holds for any finite CM-type Artin algebra with global dimension $\le 2$.
The above conclusions were first proved by Beligiannis (see \cite[Corollary 6.8 (i), (ii), (iii) and (iv)]{B2011}).
\end{remark}


\section*{Acknowledgements}
This work is supported by National Natural Science Foundation of China (Grant Nos. 11961007, 12061060, 12171207), Science Technology Foundation of Guizhou Province (Grant Nos. [2020]1Y405).




\bibliography{referLiu20221024}

\def\cprime{$'$}

\end{document}